\documentclass[11pt]{amsart}
\usepackage{amssymb, amsmath, amscd, eucal, rotating}

\input xy
 \xyoption{all}

\newif\ifPDF
\ifx\pdfoutput\undefined 
        \ifx\pdfversion\undefined
                \PDFfalse
        \fi
\else 
        \ifnum\pdfoutput > 0
                \PDFtrue
        \else
                \PDFfalse
        \fi
\fi

\numberwithin{equation}{section}

\theoremstyle{plain}

\newtheorem{proposition}{Proposition}[section]
\newtheorem{theorem}[proposition]{Theorem}		
\newtheorem{corollary}[proposition]{Corollary}
\newtheorem{lemma}[proposition]{Lemma}

\theoremstyle{definition}

\newtheorem{definition}[proposition]{Definition}
\newtheorem{remark}[proposition]{Remark}

\newcommand{\C}{\mathbb C}

\newcommand{\ZBbb}{\mathbb Z}

\newcommand{\Ccal}{{\mathcal C}}

\newcommand{\YM}{\mathop{\rm YM}\nolimits}

\newcommand{\YMH}{\mathop{\rm YMH}\nolimits}

\newcommand{\pr}{{\rm pr}}
\newcommand{\End}{\mathop{\rm End}\nolimits}

\newcommand{\vol}{\mathop{\rm vol}\nolimits}

\newcommand{\tr}{\mathop{\rm Tr}\nolimits}

\newcommand{\lra}{\longrightarrow}

\newcommand{\G}{\mathcal G}
\newcommand{\B}{\mathcal B}

\newcommand{\A}{\mathcal A}
\newcommand{\K}{\mathcal K}

\newcommand{\M}{{\mathfrak M}}

\newcommand{\Jac}{J}

\newcommand{\Bcal}{\mathcal{B}}
\newcommand{\Scal}{\mathcal{S}}
\newcommand{\Hcal}{\mathcal{H}}

\newcommand{\grad}{\nabla}

\newcommand{\Ia}{{\bf I}_a}
\newcommand{\Ib}{{\bf I}_b}
\newcommand{\IIminus}{{\bf II}^-}
\newcommand{\IIplus}{{\bf II}^+}

\newcommand{\doubleslash}{\bigr/ \negthinspace\negthinspace \bigr/}

\DeclareMathOperator{\id}{id}

\DeclareMathOperator{\rank}{rk}

\begin{document}


\title[Morse theory and stable pairs]
	{Morse theory and stable pairs}

\author[Wentworth]{Richard A.\ Wentworth}
\address{
Department of Mathematics \\
   Johns Hopkins University \\
   Baltimore, MD 21218 \\
and }
\address{
Department of Mathematics,
   University of Maryland,
   College Park, MD 20742}
\email{raw@umd.edu}
\thanks{R.W. supported in part by NSF grant DMS-0805797}

\author[Wilkin]{Graeme Wilkin}

\address{Department of Mathematics \\
		 University of Colorado \\
		Boulder, CO 80309}
\email{graeme.wilkin@colorado.edu}



\date{\today}

\begin{abstract} 
We study the Morse theory of the Yang-Mills-Higgs functional on the space of pairs
 $(A,\Phi)$, 
 where $A$ is a unitary connection on a rank 2 hermitian vector bundle over a compact Riemann
surface,  and $\Phi$ is a holomorphic section of $(E, d_A'')$.
We prove that  a certain explicitly defined substratification 
of the Morse stratification 
 is perfect in the sense of $\G$-equivariant cohomology, 
where $\G$ denotes the unitary gauge group. 
 As a consequence, 
Kirwan surjectivity holds for  pairs.  
It also follows that the twist embedding into higher degree 
 induces a surjection on equivariant
 cohomology. 
 This may be interpreted
 as a rank 2 version of the analogous statement for symmetric products of Riemann surfaces.
Finally,
we compute the $\G$-equivariant Poincar\'e polynomial of the space of $\tau$-semistable pairs. 
In particular, we recover an earlier result of Thaddeus. The analysis
provides an interpretation of the Thaddeus flips in terms of a 
variation of Morse functions.
\end{abstract}



\maketitle

\thispagestyle{empty}

\setcounter{tocdepth}{2}

\tableofcontents

\baselineskip=16pt
\setcounter{footnote}{0}
\newpage


\section{Introduction}  In this paper we revisit the notion of a stable pair on 
a Riemann surface.  We introduce new techniques for the  computation of the 
equivariant cohomology of moduli spaces.  The main ingredient is a 
 version of Morse theory
 in the spirit of Atiyah and Bott \cite{AtiyahBott83} adapted to the \emph{singular}
 infinite dimensional space of holomorphic pairs.

Recall first the basic idea.  Let $E$ be a hermitian vector bundle  over a closed Riemann surface $M$ of genus $g\geq 2$.  
The space $\A(E)$ of  unitary connections on $E$ is an infinite dimensional affine space with an action of 
the group $\G$ of unitary gauge transformations.  Via the Chern connection there is an isomorphism 
$A\mapsto d_A''$ between $\A(E)$ and the space of (integrable) Dolbeault operators (i.e.\ holomorphic structures) on $E$.  One 
of the key observations of Atiyah-Bott is that the Morse theory of a suitable $\G$-invariant functional on $\A(E)$, namely the Yang-Mills 
functional, gives rise to a smooth stratification (see also \cite{Daskal92}). Moreover, this stratification is $\G$-equivariantly perfect 
in the sense that the long exact sequences for the equivariant cohomology of successive pairs split. 
 Since $\A(E)$ is contractible, this gives an effective method, inductive on the rank of $E$, 
for computing the equivariant cohomology of the minimum, which consists of projectively flat connections.

Consider now a configuration space $\B(E)$ consisting of pairs $(A, \Phi)$, where $A\in\A(E)$ and $\Phi$ is a 
section of a vector bundle associated to $E$.
  We impose the condition that $\Phi$ be $d_A''$-holomorphic.  Note that $\B(E)$ is 
still contractible, since an equivariant  retraction of $\B(E)$ to $\A(E)$ is given by simply scaling $\Phi$.  It is 
therefore reasonable to attempt an inductive computation of equivariant cohomology as above.  A problem arises, however, 
from the singularities caused by  jumps in  the dimension of the kernel as $A$ varies. 
 Nevertheless, the methods  introduced in \cite{DWWW}  for the case of 
Higgs bundles demonstrate that in certain cases  
this difficulty can be managed.

Below we apply this approach to the moduli space of rank $2$, degree $d$, $\tau$-semistable pairs
 $\M_{\tau,d}=\B^\tau_{ss}(E)\doubleslash\G^\C$
 introduced by Bradlow \cite{Bradlow91} and Bradlow-Daskalopoulos \cite{BradlowDaskal91}.  In 
this case, $\Phi$ is holomorphic
 section of $E$, and the Yang-Mills functional $\YM(A)$  is replaced by the Yang-Mills-Higgs functional $\YMH(A,\Phi)$. We 
give a description of the algebraic and  Morse theoretic stratifications of $\B(E)$.  
 These stratifications, as well as the moduli space, depend on a real parameter $\tau$, and since $\M_{\tau,d}$ is nonempty only for $d/2<\tau<d$, 
we shall always assume this bound for $\tau$. 
 For generic $\tau$, $\G$ acts freely, and the quotient is geometric.

We will see that, as in \cite{Daskal92, DaskalWentworth04, DWWW}, the algebraic and Morse stratifications agree (see Theorem \ref{thm:morse_hn}).
Because of singularities, however, the Morse stratification actually  fails to be perfect
 in this case.  We identify
precisely how this comes about, and in fact we will show
that  this ``failure of perfection" exactly cancels between
different strata, so that there is a substratification that is indeed perfect (see Theorem
\ref{thm:perfect}).  We formulate this result as

\begin{theorem}[Equivariantly perfect stratification] \label{thm:perfection}
For every $\tau$, $d/2<\tau<d$, there is a $\G$-invariant stratification of $\mathcal{B}(E)$  
defined via  the Yang-Mills-Higgs flow that is perfect in $\G$-equivariant cohomology.
\end{theorem}
 
The fact that perfection fails for the Morse stratification but holds for a substratification seems to be a new phenomenon.
In any case, as with vector bundles, 
Theorem \ref{thm:perfection} allows us  to compute the $\G$-equivariant cohomology of 
 the open stratum
$\Bcal^\tau_{ss}(E)$. Explicit formulas in terms of symmetric products of $M$ are given in  
Theorems \ref{thm:poincare-poly-generic} and \ref{thm:poincare-poly-non-generic}.

There is a natural map (called the \emph{Kirwan map})  from the cohomology of the classifying space $B\G$ of $\G$ to the equivariant cohomology of the stratum of $\tau$-semistable pairs $\Bcal^\tau_{ss}(E)\subset\Bcal(E)$, coming from inclusion (see \cite{Kirwan84}).  One of the consequences of the work of Atiyah-Bott is that the analogous map is surjective for the case of semistable bundles.  
The same is true for pairs: 
 \begin{theorem}[Kirwan surjectivity] \label{thm:kirwan}
The Kirwan map $H^\ast(B\G)\to H^\ast_{\G}(\Bcal^\tau_{ss}(E))$ is surjective.  In particular, for generic $\tau$, 
$H^\ast(B\G)\to H^\ast(\M_{\tau,d})$ is surjective.
\end{theorem}

As noted above, for noninteger values of $\tau$, $d/2<\tau<d$, $\M_{\tau,d}$ is 
a smooth projective algebraic manifold of dimension $d+2g-2$, and the equivariant cohomology of 
$\B^\tau_{ss}(E)$ is identical to the ordinary cohomology of $\M_{\tau,d}$.  The computation of equivariant 
cohomology presented here then recovers the result of Thaddeus in \cite{Thaddeus94}, who computed the cohomology  using different methods.  Namely, 
he gives an explicit description of  the 
modifications, or ``flips", in $\M_{\tau,d}$  as the parameter $\tau$ varies.
  At integer values there is a change in stability conditions.  Below, we show how the change in cohomology arising 
from a flip may be reinterpreted as a variation of the Morse function.  This is perhaps not surprising in view of the 
construction in \cite{BradlowDaskalWentworth96}.  However, here we work directly on the infinite dimensional space.  The 
basic idea is that there is a one parameter choice of Morse functions $f_\tau$ on $\Bcal$.  The minimum $f_\tau^{-1}(0)/\G\simeq 
\M_{\tau,d}$, and the cohomology of $\M_{\tau,d}$ may, in principle, be computed from the cohomology of the higher critical sets.  
As $\tau$ varies past certain critical values, new critical sets are created while others merge.
Moreover, indices of critical sets can jump.  All this taken together accounts
for the change in topology of the minimum.
 
 There are several important points in this interpretation.  One is that the subvarieties
 responsible for the 
change in cohomology observed by Thaddeus as the parameter 
varies are somehow directly built into the Morse 
theory, even for  a fixed $\tau$, in the guise of higher critical sets. 
This example also exhibits computations at critical strata that can be carried out in
 the presence of singular normal \emph{cusps}, as 
opposed to the singular normal vector bundles in \cite{DWWW}.
 These ideas may be useful for computations in higher rank or for other moduli spaces.   
 
 The critical set corresponding to minimal Yang-Mills connections, regarded as a subset of $\B(E)$ by 
setting $\Phi\equiv 0$, is special from the point of view of the Morse theory.  In particular, essentially because of issues regarding 
Brill-Noether loci in the moduli space of vector bundles,  we can only directly prove the perfection of the 
stratification   at this step, and the crucial Morse-Bott lemma (Theorem \ref{thm:bott}), 
for $d>4g-4$.  This we do 
in Section \ref{sec:bott}.  By contrast,
for the other critical strata there is no such requirement on the degree.
 Using this fact, we then give an inductive argument  by twisting $E$ by a positive line bundle and embedding $\B(E)$ into the space of 
pairs for higher degree, thus indirectly concluding the splitting of the associated long exact sequence even at minimal Yang-Mills connections 
in low degree (see Section \ref{sec:low_degree}).

This line of reasoning leads to  another interesting consequence.  For $\tau$ close to $d/2$, there is a surjective  holomorphic map from $\M_{\tau,d}$ to
the moduli space of semistable rank 2 bundles of degree $d$.   This is  the rank 2 version of the Abel-Jacobi map 
\cite{BradlowDaskal91}.  In this sense, $\M_{\tau,d}$ is a generalization of the  $d$-th symmetric product $S^d M$ 
of $M$.  Choosing an effective divisor on $M$ of degree $k$, 
there is a natural inclusion $S^d M\hookrightarrow S^{d+k}M$, and it was shown by MacDonald in $(14.3)$ of \cite{MacDonald62} that this inclusion induces a surjection on rational 
cohomology.  A similar construction works for rank 2 pairs, except now $d\mapsto d+2k$, while there is also a shift 
in the parameter $\tau\mapsto {\tau+k}$.  We will prove  the following

\begin{theorem}[Embedding in higher degree] \label{thm:macdonald}
Let $\deg E=d$ and $\deg \widetilde E=d+2k$.  Then for all
 $d/2<\tau<d$,
the inclusion  $\B^\tau_{ss}(E)\hookrightarrow \B^{\tau+k}_{ss}(\widetilde E)$ described above
 induces a surjection on rational $\G$-equivariant cohomology.  In particular, for generic $\tau$, the inclusion $\M_{\tau,d}\hookrightarrow \M_{\tau+k, d+2k}$ induces a surjection on rational cohomology. 
\end{theorem}

\begin{remark}
It is also possible to construct a moduli space of pairs for which the isomorphism class of $\det E$ is fixed, indeed this is the space studied by Thaddeus in \cite{Thaddeus94}. The explicit calculations in this paper are all done for the non-fixed determinant case, however it is worth pointing out here that the idea is essentially the same for the fixed determinant case, and that the only major difference between the two cases is in the topology of the critical sets. In particular, the indexing set $\Delta_{\tau,d}$ for the stratification is the same in both cases.
\end{remark}

\medskip

\noindent
\emph{Acknowledgements.}  Thanks to George Daskalopoulos for many discussions.  R.W. is also grateful for the hospitality at  the MPIM-Bonn, where some of the work on this paper was completed.

\section{Stable pairs} \label{sec:pairs}

\subsection{The Harder-Narasimhan stratification}\label{subsec:Harder-Narasimhan}
Throughout this paper, $E$ will denote a rank $2$ hermitian vector bundle on $M$ of positive degree $d=\deg E$.  We will  regard $E$ as a smooth complex vector bundle, and when endowed with a holomorphic structure that is understood, we will use the same notation for the holomorphic bundle.

Recall that a holomorphic bundle $E$ of degree $d$ is \emph{stable} (resp.\ \emph{semistable}) if $\deg L< d/2$ (resp.\ $\deg L\leq d/2$) for all holomorphic line subbundles $L\subset E$.

\begin{definition} \label{def:slopes}
For a stable holomorphic bundle $E$, set $\mu_+(E)=d/2$.
For  $E$ unstable, let 
$$\mu_+(E)=\sup\{\deg L : L\subset E\text{ a holomorphic line subbundle}\}$$
For a holomorphic section $\Phi\not \equiv 0$   of $E$, define $\deg\Phi$ to be the number of zeros of $\Phi$, counted with multiplicity. Finally, for a holomorphic pair $(E,\Phi)$ let
$$
\mu_-(E,\Phi)=\begin{cases}  d-\deg \Phi & \Phi\not\equiv 0 \\  d-\mu_+(E) & \Phi\equiv 0\end{cases}
$$
\end{definition}

\begin{definition}[\cite{Bradlow91}] \label{def:taustable}
Given $\tau$, a holomorphic pair $(E,\Phi)$ is called \emph{$\tau$-stable} (resp.\ \emph{$\tau$-semistable}) if
$$
\mu_+(E)<\tau<\mu_-(E,\Phi)\qquad (\text{resp.\ } \mu_+(E)\leq\tau\leq\mu_-(E,\Phi))
$$
\end{definition}

As with holomorphic bundles, 
there is a notion of $s$-equivalence of strictly semistable objects.  
The set $\M_{\tau,d}$ of isomorphism classes of semistable pairs, 
modulo $s$-equivalence, has the structure of a projective variety. 
Note that $\M_{\tau,d}$ is empty if $\tau\not\in [d/2, d]$. 
 For non-integer values of $\tau\in (d/2, d)$, semistable is equivalent to stable, and
$\M_{\tau,d}$ is smooth. 

Let $\A=\A(E)$ denote the infinite dimensional affine space of holomorphic structures on $E$, $\G$ the group of unitary gauge transformations, and $\G^\C$ its complexification.
The space $\A$ may be identified with Dolbeault operators $A\mapsto d_A'': \Omega^0(E)\to \Omega^{0,1}(E)$, with the inverse of $d_A''$ given by the Chern connection with respect to the fixed hermitian structure.  When we want to emphasize the holomorphic bundle, we write $(E,d_A'')$.
\begin{equation}
\B =\B(E)= \left\{ (A, \Phi) \in \mathcal{A} \times \Omega^0(E) \, : \, d_A'' \Phi = 0 \right\}
\end{equation}
Let 
$$\B^{\tau}_{ss}=\left\{ (A,\Phi)\in \B : ((E,d_A''), \Phi) \text{ is $\tau$-semistable}\right\}
$$
 Then $\M_{\tau,d}=\B^{\tau}_{ss}\doubleslash\G^\C$, where the double slash identifies $s$-equivalent orbits.  For generic values of $\tau$, semistability implies stability and $\G$ acts freely, and so this is a geometric quotient.

We now describe the stratification of $\B$ associated to the Harder-Narasimhan filtration, which  has an important  relationship to the Morse theory picture that will be discussed below in Section \ref{sec:gradient}. 
 In the case of rank 2 bundles, this stratification is particularly easy to describe. For convenience, 
throughout this section we fix a generic $\tau$, $d/2<\tau<d$ (it suffices to assume $4\tau \not\in\ZBbb$). 
 Genericity  is used only to give a simple description of the strata in terms of $\delta$.  The extension to special values of $\tau$  is straightforward (see Remark \ref{rem:nongeneric}).

Note that stability of the pair fails if either of the inequalities in Definition 
\ref{def:taustable} fails.  The two inequalities are not quite independent, but 
there are some cases where only one fails and others where   both fail. If the latter,
 it seems natural to filter by the \emph{most destabilizing of the two}.
With this in mind, we make the following

\begin{definition} \label{def:defect}
For a holomorphic pair $(E,\Phi)$, let
$$
\delta(E,\Phi)=\max\left\{ \tau-\mu_-(E,\Phi), \mu_+(E)-\tau, 0\right\}
$$
\end{definition}

Note that $\delta$ takes on a discrete and infinite set of nonnegative real values, and is 
upper semicontinuous, since both $\mu_+$  and $-\mu_-$ are (observe that $\deg\Phi\leq \mu_+(E)$).  
We denote the ordered set of such $\delta$ by $\Delta_{\tau,d}$.   Clearly, $\delta$  
is an integer modulo $\pm \tau$, or $\delta =\tau-d/2$.  Because of 
the genericity of $\tau$, the former two possibilities are mutually exclusive:
\begin{lemma}
There is a disjoint union $\Delta_{\tau,d}\setminus\{0\}=\Delta_{\tau,d}^+\cup \Delta_{\tau,d}^-$, with 
\begin{align*}
\delta\in \Delta_{\tau,d}^+ &\iff \delta=\tau-\mu_-(E,\Phi)\, , \text{ for some pair } (E,\Phi) \\
\delta\in \Delta_{\tau,d}^- &\iff \delta=\mu_+(E)-\tau\, , \text{ for some pair } (E,\Phi)
\end{align*}
\end{lemma}

\begin{lemma} \label{lem:basic}
Suppose $(E,\Phi)\not \in \B^\tau_{ss}$, $\Phi\not\equiv 0$. Then
\begin{enumerate}
\item  if $\deg\Phi\geq d/2$, $\delta(E,\Phi)=\mu_+(E)-d+\tau$.
\item if $d-\tau\leq\deg\Phi<d/2$, $\delta(E,\Phi)=\deg\Phi-d+\tau$.
\item if $0\leq \deg\Phi<d-\tau$, $\delta(E,\Phi)=\mu_+(E)-\tau$.
\end{enumerate}
If $\Phi\equiv 0$, then $\delta(E,\Phi)=\mu_+(E)-d+\tau$.
\end{lemma}

\begin{proof}
If $\deg\Phi\geq d/2$, then the line subbundle generated by $\Phi$ is the 
maximal destabilizing subbundle of $E$.  Hence, $\mu_+(E)=\deg\Phi$,
 $\mu_-(E,\Phi)=d-\mu_+(E)$, and so (1) follows from the fact that $\tau>d/2$.  
For (2), consider the extension $0\to L_1\to E\to L_2\to 0$, where $\Phi\in H^0(L_1)$.  
Then $\deg L_2=d-\deg\Phi$, so $\mu_+(E,\Phi)\leq d-\deg\Phi$. It follows 
that $\mu_+(E)-\tau\leq0$.  For part (3), $0\leq \deg\Phi<d-\tau$ implies 
$\tau<\mu_-(E,\Phi)$. 
The last statement is clear, since $\tau>d/2$ implies $\tau - \mu_-(E, \Phi) 
= \mu_+(E)-d+\tau>\mu_+(E)-\tau$.
\end{proof}

\begin{corollary}
$\Delta_{\tau,d}^-\subset(0,d-\tau]$.  
\end{corollary}

\begin{proof}
Indeed, if $(E,\Phi)$ is unstable and $\delta(E,\Phi)=\mu_+(E)-\tau$, then 
by (3) it must be that $E$ is unstable and $\Phi\not\equiv 0$.  
From the Harder-Narasimhan filtration (cf.\ \cite{Kobayashi87})
 $0\to L_2\to E\to L_1\to 0$, the projection of $\Phi$ to $L_1$ must also 
be nonzero, since $\deg \Phi<\deg L_2$.  Hence, $\deg L_1=d-\mu_+(E)\geq 0$, 
and so $d - \tau \geq \delta(E, \Phi)$.
\end{proof}

\begin{remark} \label{rem:gap}
If $\delta\in \Delta^+_{\tau,d}$ and $\delta< \tau-d/2$, then $\delta\leq \tau-d/2-1/2$.
  Indeed, if $\delta+d-\tau=k\in \ZBbb$, the condition forces
$k<d/2$; hence, $k\leq d/2-1/2$.
\end{remark}

Let $I_{\tau,d}=[\tau-d/2, 2\tau-d)$. 
 We are ready to describe the $\tau$-Harder-Narasimhan stratification.
  First, for $j>d/2$, let $\A_j\subset\A$ be the set of holomorphic bundles $E$ of Harder-Narasimhan type $\mu_+(E)=j$. We also set $\A_{d/2}=\A_{ss}$.
 There is an obvious inclusion $\A_j\subset \Bcal: A\mapsto (A,0)$.

\begin{enumerate}
\item[({\bf 0})] $\delta=0$: The open stratum  $\Bcal_0^{\tau}=\B^{\tau}_{ss}$ consists of $\tau$-semistable pairs.
\item[($\Ia$)] $\delta\in\Delta_{\tau,d}^+\cap I_{\tau,d}$: Then we include the strata $\A_{\delta+d-\tau}$.
Note that this includes the semistable stratum $\A_{ss}$.  The bundles in this strata that are not semistable have a unique description as extensions
\begin{equation} \label{eqn:extension}
0\lra L_1\lra E\lra L_2\lra 0
\end{equation}
where $\deg L_1=\mu_+(E)=\delta+d-\tau$.
\item[($\Ib$)] $\delta\in \Delta_{\tau,d}^+\cap [2\tau-d,+\infty)$: Then
$
\Bcal_{\delta}^{\tau}=\left\{ (E,\Phi) : \mu_+(E)=\delta+d-\tau\right\}
$.
These are extensions \eqref{eqn:extension}, $\deg  L_1=\mu_+(E)=\delta+d-\tau$, $\Phi\subset H^0(L_1)$.
\item[($\IIplus$)] $\delta\in \Delta_{\tau,d}^+\cap (0,2\tau-d)$: Then
$
\Bcal_\delta^{\tau}=\left\{  (E,\Phi) : \deg\Phi=\delta+d-\tau\right\}
$.
These are extensions \eqref{eqn:extension}, $\deg  L_1=\delta+d-\tau$, $\Phi\subset H^0(L_1)$.
\item[($\IIminus$)] $\delta \in \Delta_{\tau,d}^-$: Then
$
\Bcal_\delta^{\tau}=\left\{ (E,\Phi) : \mu_+(E)=\delta+\tau\ ,\ \deg\Phi<d/2\right\}
$.
These  are extensions 
$$
0\lra L_2\lra E\lra L_1\lra 0
$$
where $\deg  L_2=\mu_+(E)$, and the projection of $\Phi$ to
$H^0(L_1)$ is nonzero.
\end{enumerate}
For simplicity of notation, when $\tau$ is fixed we will mostly omit the superscript: $\B_\delta=\B_\delta^\tau$.

\begin{remark}
It is simple to verify that the stratification obtained above coincides with the possible Harder-Narasimhan filtrations of pairs $(E,\Phi)$ considered as \emph{coherent systems} (see \cite{LePotier93, RV94, KingNewstead95}).
\end{remark}

It will be convenient to organize $\Delta_{\tau,d}$ by the slope  of the subbundle in the maximal destabilizing subpair.  Define $j:\Delta_{\tau,d}\setminus\{0\}\to \{d/2\}\cup\{k\in \ZBbb: k\geq d-\tau\}$ by
\begin{equation} \label{eqn:j}
j(\delta)=\begin{cases} 
\delta+d-\tau\, , & \delta\in \Delta^+_{\tau,d} \\
\delta+\tau\, , & \delta\in \Delta^-_{\tau,d}
\end{cases}
\end{equation}
Notice that $j(\delta)=\deg L_1$ for $ \delta\in \Delta^+_{\tau,d}$, and $j(\delta)=\deg L_2$ for $ \delta\in \Delta^-_{\tau,d}$,
where $L_1$, $L_2$ refer to the line subbundles of $E$ in the filtrations above.  Note that $j$ is surjective. It is precisely 2-1 on the image of $\Delta^-_{\tau,d}$ and 1-1 elsewhere (if $d$ odd; otherwise $d/2$ labels both the stratum $\A_{ss}$ and the strictly semistable bundles of type $\IIplus$).   It is not order preserving but is, of course, order preserving on each of $\Delta_{\tau,d}^\pm$ separately.

\begin{definition} \label{def:hn_stratification}
For $\delta\in \Delta_{\tau,d}$, let 
$$
X_{\delta}= \bigcup_{ \delta'\leq\delta \, ,\,
 \delta'\in  \Delta_{\tau,d}} \Bcal_{\delta'}\ \cup \bigcup_{ \delta'\leq\delta \, ,\,
 \delta'\in  \Delta_{\tau,d}^+\cap I_{\tau,d}} \A_{j(\delta')} 
 $$
For $\delta\in \Delta^+_{\tau,d}\cap I_{\tau,d}$, let
$$
 X_{\delta}'= \bigcup_{ \delta'\leq\delta \, ,\,
 \delta'\in  \Delta_{\tau,d}} \Bcal_{\delta'}\ \cup \bigcup_{ \delta'<\delta \, ,\,
 \delta'\in  \Delta_{\tau,d}^+\cap I_{\tau,d}} \A_{j(\delta')}
 $$
 For $\delta\not\in \Delta^+_{\tau,d}\cap I_{\tau,d}$, let 
 $$
 X_{\delta}'= \bigcup_{ \delta'<\delta \, ,\,
 \delta'\in  \Delta_{\tau,d}} \Bcal_{\delta'}\ \cup \bigcup_{ \delta'<\delta \, ,\,
 \delta'\in  \Delta_{\tau,d}^+\cap I_{\tau,d}} \A_{j(\delta')}
 $$
 We call the collection $\{X_\delta, X_\delta'\}_{\delta\in \Delta_{\tau, d}}$ the \emph{$\tau$-Harder-Narasimhan stratification} of $\B$.
\end{definition}
\noindent
Note that $X_{\delta_1}\subset X_\delta'\subsetneq X_\delta\subset X_{\delta_2}'$, where $\delta_1$ is the predecessor and $\delta_2$ is the successor of $\delta $ in $\Delta_{\tau,d}$.  If $\delta\not\in\Delta^+_{\tau,d}\cap I_{\tau,d}$, then $X_\delta'=X_{\delta_1}$ and $X_\delta=X_{\delta_2}'$ .    In the special case $\delta=\tau-d/2$, we  have
\begin{align}
X_{\tau-d/2}&= X'_{\tau-d/2}\cup \A_{ss}   \label{eqn:x_xprime} \\
X'_{\tau-d/2}&=
\begin{cases}  X_{\delta_1} &\text{ if $d$ is odd} \\
X_{\delta_1}\cup \B_{\tau-d/2} &\text{ if $d$ is even}
\end{cases} \label{eqn:xprime}
\end{align}
The following is clear.
\begin{proposition}
The sets $\{X_\delta, X_\delta'\}_{\delta\in \Delta_{\tau, d}}$, are locally closed in $\B$, $\G$-invariant, and satisfy
\begin{align*}
\B=\bigcup_{\delta\in \Delta_{\tau, d}} X_\delta&=\bigcup_{\delta\in \Delta_{\tau, d}} X_\delta' \\
\overline \B_\delta\subset\bigcup_{\delta\leq \delta'\, ,\, \delta'\in \Delta_{\tau, d}} 
\B_{\delta'} 
&=\B_\delta\cup \bigcup_{\delta<\delta'\, ,\, \delta'\in \Delta_{\tau, d}} \B_{\delta'}'  
\\ 
\overline \B_\delta'\subset \bigcup_{\delta\leq \delta'\, ,\, \delta'\in \Delta_{\tau, d}}
 \B_{\delta'}'
&=\B_\delta'\cup \bigcup_{\delta<\delta'\, ,\, \delta'\in \Delta_{\tau, d}} \B_{\delta'}
\end{align*}
\end{proposition}

\begin{remark} \label{rem:nongeneric}
To extend this stratification in the case of nongeneric $\tau$, we define 
the sets $\Delta_{\tau,d}^\pm$ and the corresponding strata as 
above.  For   $\delta\in \Delta_{\tau,d}^+\cap\Delta_{\tau,d}^-$ there are two or possibly
three  components with the same label.
\end{remark}

Let us note the following behavior as $\tau$ varies. 
 For $\tau_1\leq \tau_2$, 
there is a well-defined map $\Delta_{\tau_1,d}\to \Delta_{\tau_2,d}$
given by $\delta\mapsto\max\{\delta\pm (\tau_2-\tau_1), 0\}$, where $\pm$ depends on 
$\delta\in \Delta^\pm_{\tau,d}$. 
Hence,  elements of $\Delta^+_{\tau,d}$ (white circles in Figure 1 below) 
``move" to the right, and elements of $\Delta^-_{\tau,d}$ (dark circles) 
 ``move" to the left as $\tau $ increases.
 The map is an order preserving bijection \emph{provided}
 $\tau_1$, $\tau_2$ are in a connected component of
$(d/2, d)\setminus C_d$, where
\begin{equation} \label{eqn:cd}
C_d=\{ \tau_c\in (d/2,d) : 2\tau_c\in \ZBbb \text{ if $d$ even}, 4\tau_c\in 
\ZBbb \text{ if $d$ odd}\}
\end{equation}
However, as $\tau_2$ crosses an element of $C_d$, there is a ``flip" in the stratification.  
When this flip occurs at $\delta=0$, 
this is the phenomenon discovered by Thaddeus \cite{Thaddeus94};
 the discussion here is an extension  of this effect to the entire
stratification.

\setlength{\unitlength}{1cm}
\begin{picture}(14,6)
\put(2.5,4.9){$\Delta_{\tau_1,d}$}
\put(2.5,2.4){$\Delta_{\tau_2,d}$}
\ifPDF
\put(3.6,1){
{\scalebox{.5}{\includegraphics{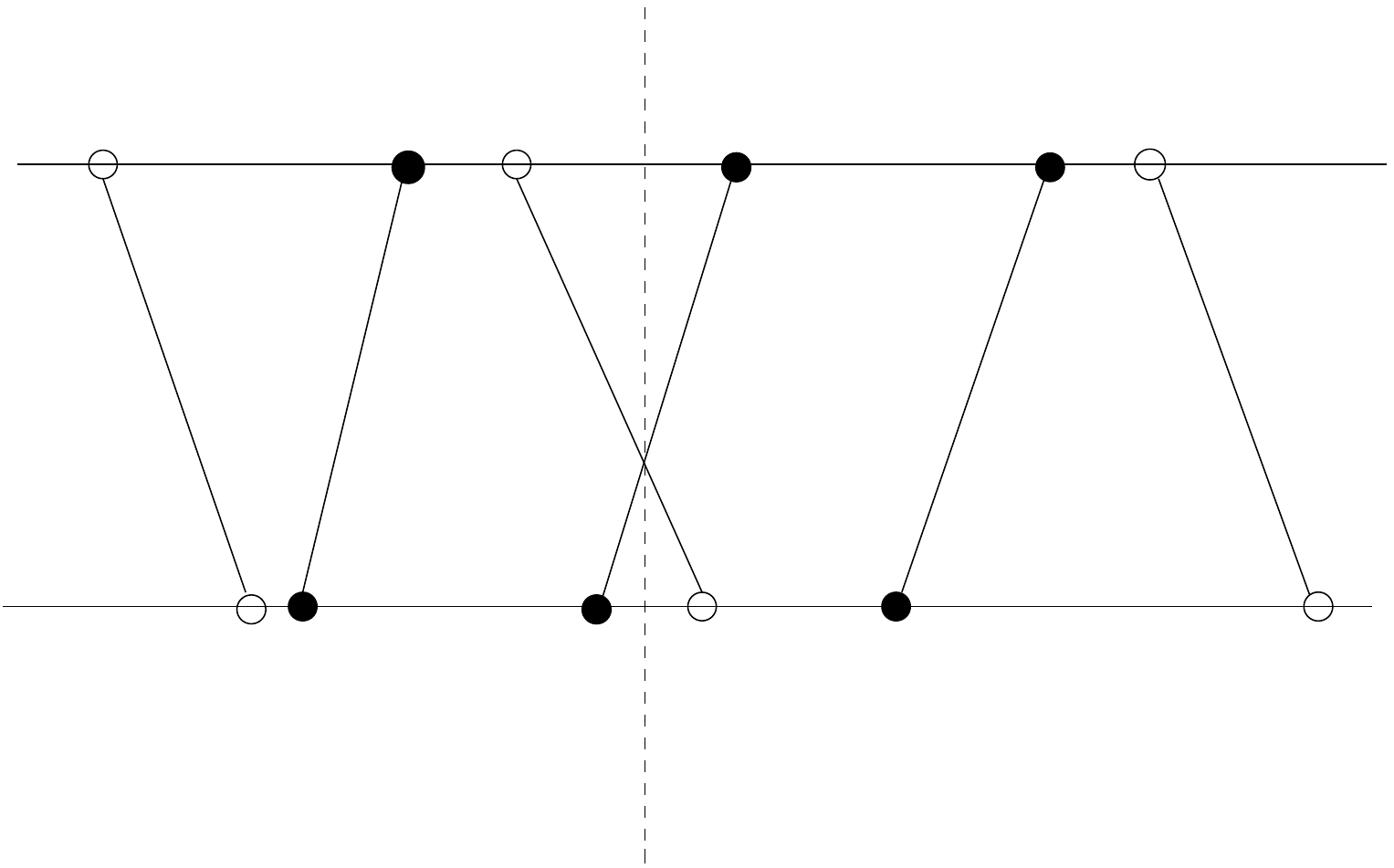}}}}
\else
\put(3.6,1){
{\scalebox{.5}{\includegraphics{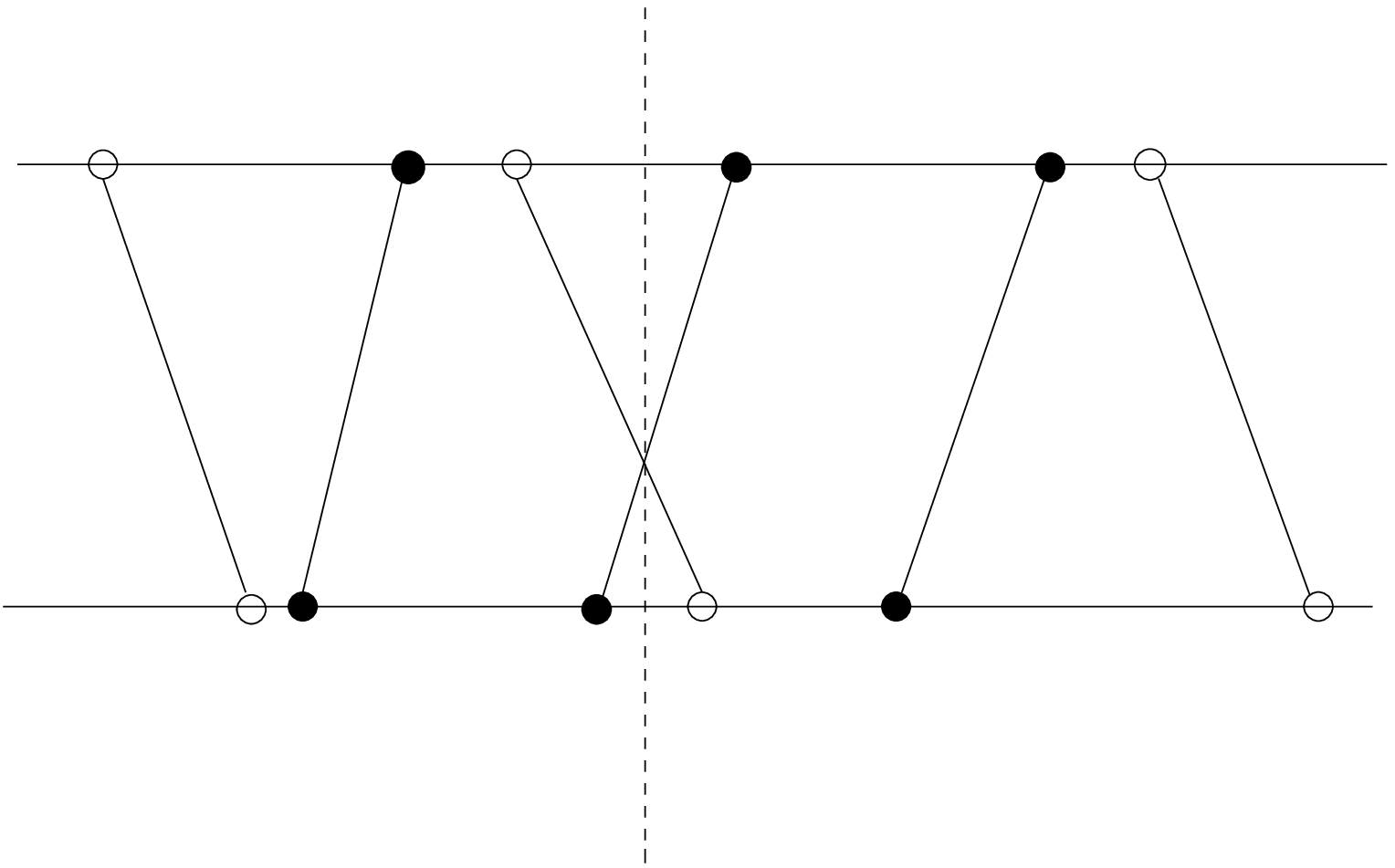}}}}
\fi
\put(6,0){Figure 1. A ``flip"}
\end{picture}
\medskip

Finally, we will also have need to refer to the Harder-Narasimhan stratification of the space $\A$ of unitary connections on $E$.  We denote this by
\begin{equation} \label{eqn:Ahn}
X^\A_j=\bigcup_{d/2\leq j'\leq j} \A_{j'}
\end{equation}
The following statement   will be used later on.  It is  an immediate consequence of the descriptions of the strata above.

\begin{lemma} \label{lem:projection}
Consider the projection $\pr : \B\to \A$.  Then
\begin{align*}
\pr(\Bcal_\delta)&= \A_{j(\delta)}\ , \  \delta\in \Delta^-_{\tau,d}\cup\left( \Delta^+_{\tau,d}\cap[\tau-d/2,+\infty)\right)  \\
\pr(\Bcal_\delta)&= X^\A_{d-j(\delta)}\, ,\ \delta\in \Delta^+_{\tau,d}\cap (0,\tau-d/2) 
\end{align*}
\end{lemma}

\subsection{Deformation theory}
 Fix a conformal metric on $M$, normalized\footnote{More generally, the scale invariant parameter is $\tau\vol(M)/2\pi$.} for convenience so that $\vol(M)=2\pi$.
Infinitesimal deformations of $(A, \Phi) \in \mathcal{B}$ modulo equivalence are described by the following elliptic complex, which we denote by $\mathcal{C}_{(A, \Phi)}$ (cf.\ \cite{BradlowDaskal91}).

\begin{align} 
\begin{split}
\xymatrix{
\mathcal{C}_{(A, \Phi)}^0 \ar[r]^(0.45){D_1} &\mathcal{C}_{(A, \Phi)}^1\ar[r]^(0.45){D_2} & \mathcal{C}_{(A, \Phi)}^2
} \hskip.75in
\label{eqn:deformation-complex}\\
\xymatrix{
\Omega^0(\End E) \ar[r]^(0.35){D_1} & \Omega^{0,1}(\End E) \oplus \Omega^{0}(E) \ar[r]^(0.65){D_2} & \Omega^{0,1}(E) 
}
\end{split}
\end{align}
$$
D_1(u)  = (-d_A'' u, u\Phi)
\ ,\
D_2(a, \varphi)  = d_A'' \varphi + a \Phi
$$
Here, $D_1$ is the linearization of the action of the complex gauge group $\G^\C$ on $\mathcal{B}$, and $D_2$ is the linearization of the condition $d_A''\Phi=0$.
Note that $D_2 D_1  = 0$ if $(A, \Phi) \in \mathcal{B}$. 
The hermitian metric gives adjoint operators 
\begin{equation} \label{eqn:adjoint}
D_1^\ast(a,\varphi)=-(d_A'')^\ast a+\varphi\Phi^\ast\, ,\,
D_2^\ast(\beta)=(\beta\Phi^\ast, (d_A'')^\ast\beta)
\end{equation}
The spaces of \emph{harmonic forms} are  by definition
\begin{align*}
{\mathcal H}^0(\mathcal{C}_{(A, \Phi)})&=\ker D_1 \\
{\mathcal H}^1(\mathcal{C}_{(A, \Phi)})&=\ker D_1^\ast\cap \ker D_2 \\
{\mathcal H}^2(\mathcal{C}_{(A, \Phi)})&=\ker D_2^\ast 
\end{align*}

Vectors in $\Omega^{0,1}(\End E) \oplus \Omega^{0}(E)$ that are orthogonal to the $\mathcal{G}^\C$-orbit through $(A, \Phi)$ are in $\ker D_1^*$, and a \emph{slice} for the action of $\mathcal{G}^\C$ on $\mathcal{B}$ is therefore given by
\begin{equation}\label{eqn:def-slice}
\mathcal{S}_{(A, \Phi)} = \ker D_1^* \cap \left\{ (a, \varphi) \in\Omega^{0,1}(\End E) \oplus \Omega^{0}(E) \, : \, D_2(a,\varphi) + a\varphi =0\right\}
\end{equation}
Define the \emph{slice map} 
\begin{align}\label{eqn:local-homeo}
\begin{split}
\Sigma : &\left( \ker D_1 \right)^\perp \times \mathcal{S}_{(A, \Phi)}  \rightarrow \mathcal{B} \\
 (u, a, \varphi) & \mapsto e^u \cdot ( A + a, \Phi + \varphi) 
\end{split}
\end{align}

The proof of the following may be modeled on  \cite[Proposition 4.12]{Wilkin08}.  We omit the details.
\begin{proposition}\label{prop:slice}
The slice map $\Sigma$ is a local homeomorphism from a neighborhood of $0$ in $(\ker D_1)^\perp \times \mathcal{S}_{(A, \Phi)}$ to a neighborhood  of $(A, \Phi)$ in $\mathcal{B}$.
\end{proposition}

The \emph{Kuranishi map} is defined by
\begin{align*}
 \Omega^{0,1}(\End E ) \oplus \Omega^{0}(E)&{ \buildrel k\over \longrightarrow }\ \Omega^{0,1}(\End E) \oplus \Omega^{0}(E)\notag \\
k(a, \varphi) & = (a, \varphi) + D_2^* \circ G_2 (a\varphi) 
\end{align*}
where $G_2$ denotes the Green's operator associated to the laplacian $D_2(D_2)^\ast$.
We have the following standard result (cf.\ \cite[Chapter VII]{Kobayashi87} for the case of holomorphic bundles over a K\"ahler manifold and \cite{BradlowDaskal91} for this case).
\begin{proposition} \label{prop:kuranishi}
The Kuranishi map $k$ maps $\Scal_{(A,\Phi)}$ to harmonics $\Hcal^1(\Ccal_{(A, \Phi)})$, and
in a neighborhood of zero  it is
 a local homeomorphism   onto its image. Moreover, if $\Hcal^2(\Ccal_{(A, \Phi)})=\{0\}$, then $k$ is a local homeomorphism $\Scal_{(A,\Phi)}\to\Hcal^1(\Ccal_{(A, \Phi)})$.
\end{proposition}

The following is immediate from \eqref{eqn:deformation-complex} and \eqref{eqn:adjoint}.
\begin{lemma} \label{lem:h2}
Given $(A,\Phi)\in \B$, if $\Phi\neq 0$ then $\Hcal^0(\mathcal{C}_{(A, \Phi)})=\Hcal^2(\mathcal{C}_{(A, \Phi)})=\{0\}$.
If $ H^1(E)=\{0\}$ then $\Hcal^2(\mathcal{C}_{(A, \Phi)})=\{0\}$.
\end{lemma}

We will be interested in the deformation complex along higher critical sets 
of the Yang-Mills-Higgs functional.  As we will see in the next section,
 in addition to the Yang-Mills connections (where $\Phi\equiv 0$), the 
other critical sets correspond to split bundles  $E = L_1 \oplus L_2$,  
$(A,\Phi)=(A_1\oplus A_2, \Phi_1\oplus \{0\})$,  with $\deg L_1=j\geq\deg L_2=d-j$.  
Here, $j=j(\delta)$ for some $\delta\in \Delta_{\tau,d}^+$, or $j=d-j(\delta)$ 
for some $\delta\in \Delta_{\tau,d}^-$.
 The set of all such critical points will therefore be denoted by $\eta_\delta\subset
\B$.   We will denote the  components of $\End E \simeq L_i\otimes L_j^\ast$ 
in the complex by $u_{ij}$, $a_{ij}$, $\varphi_{ij}$.

In this case, ${\mathcal H}^1({\mathcal C}_{(A,\Phi)})$ consists of all $(a,\varphi)$ satisfying
\begin{eqnarray}
(d_A'')^\ast a_{12}=0 & (d'')^\ast a_{22}=0  \notag\\
(d'')^\ast a_{11}-\varphi_1\Phi_1^\ast=0 & (d_A'')^\ast a_{21}-\varphi_2\Phi_1^\ast=0 \label{eqn:h1} \\
d_{A_1}''\varphi_{1}+a_{11}\Phi_1=0 &
d_{A_2}'' \varphi_{2}+a_{21}\Phi_1=0\notag
\end{eqnarray}

We use this formalism to define deformation retractions in a neighborhood of $(A,\Phi)\in\B$ in two cases.  First, we have
\begin{lemma} \label{lem:retraction1}
Suppose $(A,\Phi)=(A_1\oplus A_2,\Phi_1\oplus 0)$ is a split pair as above, $\Phi_1\neq 0$.
 Let 
\begin{align*}
\Scal_{(A,\Phi)}^{neg.}&=\{ (a,\varphi)\in \Scal_{(A,\Phi)} : 
a_{ij}=0\, ,\, (ij)\neq (21)\, , \text{and } \varphi_1=0\} \\ 
 \Scal_{(A,\Phi)}'&=\{ (a,\varphi)\in \Scal_{(A,\Phi)} : (a_{21},\varphi_2)\neq0\}
\end{align*}
Then there is an equivariant deformation retraction $\Scal_{(A,\Phi)}^{neg.}
\hookrightarrow \Scal_{(A,\Phi)}$ which restricts to a deformation retraction 
$\Scal_{(A,\Phi)}^{neg.}\setminus\{0\}\hookrightarrow \Scal_{(A,\Phi)}'$.
\end{lemma}

\begin{proof}
By Lemma \ref{lem:h2} and Proposition \ref{prop:kuranishi}, the Kuranishi map gives a homeomorphism of the slice with $\Hcal^1(\Ccal_{(A,\Phi)})$. 
Hence, it suffices to  define the retraction there.  For this we take
 $$
 r_t(a_{11}, a_{12}, a_{21}, a_{22};\varphi_1, \varphi_2)=(ta_{11}, ta_{12}, a_{21}, ta_{22};t\varphi_1, \varphi_2) \, ,\, t\in [0,1]
 $$
 Notice that this preserves the equations in \eqref{eqn:h1}.
\end{proof}

Second, near minimal Yang-Mills connections, we find a similar retraction under the assumption that $\Hcal^2(\mathcal{C}_{(A, \Phi)})$ vanishes.
\begin{lemma} \label{lem:retraction2}
Suppose $d>4g-4$ and $A$ is semistable.  Let 
\begin{align*}
\Scal_{(A,0)}^{neg.}&=\{ (a,\varphi)\in \Scal_{(A,0)} : a=0\} \\
\Scal_{(A,0)}'&=\{ (a,\varphi)\in \Scal_{(A,0)} : \varphi\neq0\}
\end{align*}
Then there is an equivariant deformation retraction $\Scal_{(A,0)}^{neg.}
\hookrightarrow \Scal_{(A,0)}$ 
which restricts to a deformation retraction $\Scal_{(A,0)}^{neg.}
\setminus\{0\}\hookrightarrow \Scal_{(A,0)}'$.
\end{lemma}

\begin{proof}  
Let $E$ be the holomorphic bundle given by $A$.  Since $E$ is semistable, so is $E^\ast\otimes K_M$, where $K_M$ is the canonical bundle of $M$.  On the other hand, by 
the assumption, $\deg(E^\ast\otimes K_M)=4g-4-\deg E<0$.  Hence, by Serre duality,  $H^1(E)
\simeq H^0(E^\ast\otimes K_M)^\ast=\{0\}$.
Given $a$, let ${\mathcal H}_a$ denote harmonic projection to $\ker d_{A+a}''$.  
It follows that
 for $a$ in a small neighborhood of the origin in the slice, ${\mathcal H}_a$ is a continuous family.  We can therefore define the deformation retraction explicitly by
 $$
 r_t(a,\varphi)= (ta, {\mathcal H}_{ta}(\varphi))\, ,\, t\in [0,1]
 $$
 For a sufficiently small neighborhood of the origin in the slice, this preserves the set 
$\Scal_{(A,0)}'$.  It is also clearly equivariant.
\end{proof}

\section{Morse theory}

\subsection{The $\tau$-vortex equations}\label{subsec:vortex-eqns}

Let $\mu(A, \Phi) = *F_A - i  \Phi \Phi^*$. 
Then $*\mu$ is a moment map for the action of $\mathcal{G}$ on 
$\mathcal{B} \subset \mathcal{A} \times \Omega^0(E)$. 
Let $\tau > 0$ be a positive parameter and define the Yang-Mills-Higgs functional 
\begin{equation}
f_\tau(A, \Phi) = \| \mu + i \tau \cdot \id \|^2 
\end{equation}
Solutions to the  \emph{$\tau$-vortex equations} are the absolute minima of $f_\tau$:
\begin{equation}\label{eqn:tau-vortex}
\mu(A, \Phi) + i \tau \cdot \id = 0 
\end{equation}

\begin{theorem}[Bradlow \cite{Bradlow91}]\label{thm:Hitchin-Kobayashi}
$
\M_{\tau, d} = \left\{ (A, \Phi) \in \mathcal{B} \, : \, \mu(A, \Phi) + i \tau \cdot \id = 0 \right\} / \mathcal{G}
$.
\end{theorem}

If the space of solutions to the $\tau$-vortex equations is nonempty, then $\tau$ must satisfy the following restriction.
\begin{align}\label{eqn:tau-restriction}
\begin{split}
\mu + i \tau \cdot \id = * F_A - i \Phi \Phi^* &+ i \tau \cdot \id  = 0 \\
\Longrightarrow \quad \frac{i}{2 \pi} \int_M \tr( * F_A - i \Phi \Phi^* )  = 2 \tau \
&\Longleftrightarrow \quad \deg E + \| \Phi \|^2  = 2\tau  
\end{split}
\end{align}
Therefore $2\tau \geq d$ (with strict inequality if we want to ensure that $\Phi \neq 0$). 
Theorem 2.1.6 of \cite{Bradlow91} shows that a solution
 to the $\tau$-vortex equations which is not $\tau$-stable must split. 
Moreover, since $\rank E = 2$  the solutions can only split if 
$\tau$ is an integer. In particular, 
for a generic choice of $\tau$ solutions to \eqref{eqn:tau-vortex} must be $\tau$-stable.
In general, 
critical sets of
  $f_\tau$ can be characterized in terms of a decomposition of the holomorphic structure of $E$.
The critical point equations for the functional $f_\tau$ are
\begin{align}
d_A'' \left( \mu + i \tau \cdot \id \right) & = 0 \label{eqn:a-critical} \\
\left( \mu + i \tau \cdot \id \right) \Phi  & = 0 \label{eqn:phi-critical}
\end{align}
There are three different types of critical points.

\begin{enumerate}
\item[({\bf 0})]  Absolute minimum $f_\tau^{-1}(0)$.
\item[({\bf I})] \label{item:yang_mills}  Yang-Mills connections with
 $\Phi = 0$.  Then either
$A$ is an irreducible Yang-Mills minimum or $E$ splits holomorphically 
as $E = L_1 \oplus L_2$. The latter exist for all values of $\deg L_1 \geq d/2$ 
and the existence of the critical points is independent of the choice of 
$\tau$. However, as shown below the Morse index does depend on $\tau$.   
If $E$ is semistable (resp.\ $\deg L_1<\tau$) we call this a critical point 
of type $\Ia$, and we label it $\delta=\tau-d/2$ (resp.\ $\delta=\deg L_1-d+\tau$). 
 If $\deg L_1>\tau$ it is of type $\Ib$, and set $\delta=\deg L_1-d+\tau$.
\item[({\bf II})] \label{item:split-vortex} $E$ splits holomorphically as $E =L_1 \oplus L_2$, and $\Phi \in H^0(L_1) \setminus \{ 0\}$. On $L_1$ we have 
$$
*F_{A_1} - i \Phi \Phi^* = -i\tau \ ,\
\| \Phi \|^2 = 2\pi(\tau- \deg L_1)
$$
Therefore $\deg L_1< \tau$.  
Further subdivide these depending upon $\deg L_1$.  
\begin{enumerate}
\item[($\IIminus$)] $\deg L_1\leq d-\tau$, $\delta=d-\deg  L_1-\tau $;
\item[($\IIplus$)] $d-\tau<\deg  L_1<\tau$,  $\delta=\deg L_1-d+\tau$;
\end{enumerate}

\end{enumerate}

Let $S^dM$ denote the $d$-th symmetric product of the Riemann surface $M$, and $J_d(M)$ the Jacobian variety of degree $d$ line bundles on $M$.
For future reference we record the following
\begin{proposition} \label{prop:critical_set_cohomology}
For $\delta\in \Delta_{\tau,d}\setminus\{0\}$,
$$
H^\ast_\G(\eta_\delta)=\begin{cases}
H^\ast_\G(\A_{ss})&\text{ Type {\bf I}, $\delta=\tau-d/2$} \\
H^\ast(J_{j(\delta)}(M)\times J_{d-j(\delta)}(M))\otimes H^\ast(BU(1)\times BU(1)) &\text{ Type {\bf I}, $\delta\neq\tau-d/2$} \\
H^\ast(S^{j(\delta)}M\times J_{d-j(\delta)}(M))\otimes H^\ast (BU(1))&\text{ Type} \, \, \IIplus \\
H^\ast(S^{d-j(\delta)}M\times J_{j(\delta)}(M))\otimes H^\ast (BU(1))&\text{ Type} \, \, \IIminus
\end{cases}
$$
\end{proposition}

\subsection{The gradient flow} \label{sec:gradient}
Consider the negative gradient flow of the Yang-Mills-Higgs functional
$f_\tau$ defined on the space $\mathcal{B} \subset \mathcal{A} \times 
\Omega^0(E)$. Since the functional is very similar to that studied in
 \cite{Hong01},  we only sketch the details of the existence and 
convergence of the flow and focus on showing that the Morse stratification induced 
by the flow is equivalent to the Harder-Narasimhan stratification described 
in Section \ref{subsec:Harder-Narasimhan}.

The  gradient flow equations  are
\begin{equation}\label{eqn:grad-flow-eqns}
\frac{\partial A}{\partial t}  = 2*d_A (\mu + i \tau) \ ,\
\frac{\partial \Phi}{\partial t}  = -4i (\mu + i \tau) \Phi
\end{equation}

\begin{theorem}
The gradient flow of $f_\tau$ with initial conditions in 
$\mathcal{B}$ exists for all time and converges to a critical point of $f_\tau$ in the smooth topology.
\end{theorem}

A standard calculation (cf. \cite[Section 4]{Bradlow91}) shows that $f_\tau$ can be re-written as
\begin{equation}
f_\tau = \int_X \left( \left| F_A \right|^2 + \left| d_A' \Phi \right|^2 + \left| \Phi \Phi^* \right|^2 - 2 \tau \left| \Phi \right|^2 + \left| \tau \right|^2 \right) dvol + 4 \tau \deg E 
\end{equation}

This is very similar to the functional $\YMH$ studied in \cite{Hong01}, and 
the proof for existence of the flow for all positive time follows the same 
structure (which is in turn modeled on Donaldson's proof for the Yang-Mills functional 
in \cite{Donaldson85}), therefore we omit the details. An important part 
of the proof worth mentioning here is that the flow is generated by 
the action of $\mathcal{G}^\C$, i.e. for all $t \in [0, \infty)$ there 
exists $g(t) \in \mathcal{G}^\C$ such that the solution $(A(t), \Phi(t))$ 
to the flow equations \eqref{eqn:grad-flow-eqns} with initial condition $(A, \Phi)$ is 
given by $(A(t), \Phi(t)) = g(t) \cdot (A, \Phi)$.

To show that the gradient flow converges, one can use the results 
of Theorem B of \cite{HongTian04} (where again, the functional is not exactly 
the same as $f_\tau$, but it has the same structure and so the proof of 
convergence is similar). The statement of \cite[Theorem B]{HongTian04} only describes 
smooth convergence along a subsequence (since they also study the higher dimensional 
case where bubbling occurs), and to extend this to show that the limit is 
unique we use the Lojasiewicz inequality technique of \cite{Simon83} and \cite{Rade92}.
 The key estimate is contained in the following proposition.

\begin{proposition}
Let $(A_\infty, \Phi_\infty)$ be a critical point of $f_\tau$. Then there exist $\varepsilon_1 > 0$, a positive constant $C$, and $\theta \in \left(0, \frac{1}{2} \right)$, such that $\left\| (A, \Phi) - (A_\infty, \Phi_\infty) \right\|$ implies that
\begin{equation}
\left\| \grad f_\tau (A, \Phi) \right\|_{L^2} \geq C \left| f_\tau(A, \Phi) - f_\tau(A_\infty, \Phi_\infty) \right|^{1-\theta} 
\end{equation} 
\end{proposition}
\noindent
The proof is similar to that in \cite{Wilkin08}, and so is omitted. 

The rest of the proof of convergence then follows the analysis in \cite{Wilkin08} for Higgs bundles. The key result is the following proposition, which is the analog of \cite[Proposition 3.7]{Wilkin08} (see also \cite{Simon83} or \cite[Proposition 7.4]{Rade92}).
\begin{proposition}\label{prop:flow-below}
Each critical point $(A, \Phi)$ of $f_\tau$ has a neighborhood $U$ such that if $\left( A(t), \Phi(t) \right)$ is a solution of the gradient flow equations for $f_\tau$ and $\left( A(T), \Phi(T) \right) \in U$ for some $T$, then either $f_\tau \left( A(t), \Phi(t) \right) < f_\tau \left( A, \Phi \right)$ for some $t$, or $\left( A(t), \Phi(t) \right)$ converges to a critical point $(A', \Phi')$ such that $f_\tau (A', \Phi') = f_\tau(A, \Phi)$. Moreover, there exists $\varepsilon$ (depending on $U$) such that $\| (A', \Phi') - (A, \Phi) \| < \varepsilon$.
\end{proposition}

The next step is the main result of this section: The Morse stratification
 induced by the gradient flow of $f_\tau$ is the same as the
 $\tau$-Harder-Narasimhan stratification described in
 Section \ref{subsec:Harder-Narasimhan}. First recall the Hitchin-Kobayashi correspondence from Theorem \ref{thm:Hitchin-Kobayashi}, and the distance-decreasing result from \cite{Hong01}, which can be re-stated as follows.

\begin{lemma}[Hong \cite{Hong01}]\label{lem:distance-decreasing}
Let $(A_1, \Phi_1)$ and $(A_2, \Phi_2)$ be two pairs related by an element $g \in \mathcal{G}^\C$. Then the distance between the $\mathcal{G}$-orbits of $(A_1(t), \Phi_1(t))$ and $(A_2(t), \Phi_2(t))$ is non-increasing along the flow.
\end{lemma}

Recall that the critical sets associated to each stratum are given in Section \ref{subsec:vortex-eqns}, and that the critical set associated to the stratum $\mathcal{B}_\delta$ is denoted $\eta_\delta$. Define $S_\delta \subset \mathcal{B}$ to be the subset of pairs that converge to a point in $C_\delta$ under the gradient flow of $f_\tau$.
The next lemma gives some standard results about the critical sets of $f_\tau$.

\begin{lemma}\label{lem:background-results}
\begin{enumerate}

\item The critical set $\eta_\delta$ is the minimum of the functional $f_\tau$ on the stratum $\mathcal{B}_\delta$.

\item \label{item:closure} The closure of each $\mathcal{G}^\C$ orbit in $\mathcal{B}_\delta$ intersects the critical set $\eta_\delta$.

\item There exists $\varepsilon > 0$ (depending on $\tau$) such that $(A, \Phi) \in \eta_\delta$ and $(A', \Phi') \in \eta_{\delta'}$ with $\delta \neq \delta'$ implies that $\| (A, \Phi) - (A', \Phi') \| \geq \varepsilon$.

\end{enumerate}
\end{lemma}

\begin{proof}
Since these results are analogous to standard results for the Yang-Mills functional
 (see for example \cite{AtiyahBott83}, \cite{Daskal92}, or \cite{DaskalWentworth04}), 
and the proof for holomorphic pairs is similar,  we only sketch the idea of the proof here.
\begin{itemize}
\item
The first statement follows by noting that the convexity of the norm-square 
function $\| \cdot \|^2$ shows that the minimum of $f_\tau$ on each extension 
class occurs at a critical point. This can be checked explicitly for each 
of the types $\Ia$, $\Ib$, $\IIplus$, and $\IIminus$.
\item
To see the second statement, simply scale the extension class and apply Theorem \ref{thm:Hitchin-Kobayashi} (the Hitchin-Kobayashi correspondence) to the graded object of the filtration (cf. \cite[Theorem 3.10]{DaskalWentworth04} for the Yang-Mills case).
\item
The third statement can be checked by noting that (modulo the $\mathcal{G}$-action) the 
critical sets are compact, and then explicitly computing the 
distance between distinct critical sets.
\end{itemize}
\end{proof}

As a consequence we have

\begin{proposition}\label{prop:grad-flow-results}

\begin{enumerate}

\item \label{item:nbhd}  Each critical set $\eta_\delta$ has a neighborhood $V_\delta$ such that $V_\delta \cap \mathcal{B}_\delta \subset S_\delta$.

\item \label{item:group-invariance} $S_\delta \cap \mathcal{B}_\delta$ is $\mathcal{G}^\C$-invariant.

\end{enumerate}
\end{proposition}

\begin{proof}

Proposition \ref{prop:flow-below} implies that there exists a neighborhood $V_\delta$ of each critical set $\eta_\delta$ such that if $(A, \Phi) \in V_\delta$ then the flow with initial conditions $(A, \Phi)$ either flows below $\eta_\delta$, or converges to a critical point close to $\eta_\delta$. Since $f_\tau$ is minimized on each Harder-Narasimhan stratum $\mathcal{B}_\delta$ by the critical set $\eta_\delta$, the flow is generated by the action of $\mathcal{G}^\C$, and the strata $\mathcal{B}_\delta$ are $\mathcal{G}^\C$-invariant, then the first alternative cannot occur if $(A, \Phi) \in \mathcal{B}_\delta \cap V_\delta$. Since the critical sets are a finite distance apart, then (by shrinking $V_\delta$ if necessary) the limit must be contained in $\eta_\delta$. Therefore $V_\delta \cap \mathcal{B}_\delta \subset S_\delta$, which completes the proof of \eqref{item:nbhd}.

To prove \eqref{item:group-invariance}, for each pair $(A, \Phi) \in S_\delta \cap \mathcal{B}_\delta$, let $Y_{(A, \Phi)} = \left\{ g \in \mathcal{G}^\C \, : \, g \cdot (A, \Phi) \in S_\delta \cap \mathcal{B}_\delta \right\}$. The aim is to show that $Y_{(A, \Phi)} = \mathcal{G}^\C$. 
Firstly we note that since the group $\Gamma$ of components of $\mathcal{G}^\C$ is the same as that for the unitary gauge group $\mathcal{G}$, the flow equations \eqref{eqn:grad-flow-eqns} are $\mathcal{G}$-equivariant, and the critical sets $\eta_\delta$ are $\mathcal{G}$-invariant, then it is sufficient to consider the connected component of $\mathcal{G}^\C$ containing the identity. Therefore the problem reduces to showing that $Y_{(A, \Phi)}$ is open and closed. Openness follows from the continuity of the group action, the distance-decreasing result of Lemma \ref{lem:distance-decreasing}, and the result in part \eqref{item:nbhd}. Closedness follows by taking a sequence of points $\{ g_k \} \subset Y_{(A, \Phi)}$ that converges to some $g \in \mathcal{G}^\C$, and observing that the distance-decreasing result of Lemma \ref{lem:distance-decreasing} implies that the flow with initial conditions $g \cdot (A, \Phi)$ must converge to a limit close to the $\mathcal{G}$-orbit of the limit of the flow with initial conditions $g_k \cdot (A, \Phi)$ for some large $k$. Since the critical sets are $\mathcal{G}$-invariant, and critical sets of different types are a finite distance apart, then by taking $k$ large enough (so that $g_k \cdot (A, \Phi)$ is close enough to $g \cdot (A, \Phi)$) we see that the limit of the flow with initial conditions $g \cdot (A, \Phi)$ must be in $\eta_\delta$. Therefore $Y_{(A, \Phi)}$ is both open and closed.
\end{proof}

\begin{theorem} \label{thm:morse_hn}
The Morse stratification by gradient flow is the same as the Harder-Narasimhan stratification in Definition \ref{def:hn_stratification}.
\end{theorem}

\begin{proof}
The goal is to show that $\mathcal{B}_\delta \subseteq S_\delta$ for each $\delta$. Let $x \in \mathcal{B}_\delta$. By Lemma \ref{lem:background-results} \eqref{item:closure} the closure of the orbit $\mathcal{G}^\C \cdot x$ intersects $\eta_\delta$, therefore there exists $g \in \mathcal{G}^\C$ such that $g \cdot x \in V_\delta \cap \mathcal{B}_\delta \subseteq S_\delta$ by Proposition \ref{prop:grad-flow-results} \eqref{item:nbhd}. Since $S_\delta \cap \mathcal{B}_\delta$ is $\mathcal{G}^\C$-invariant by Proposition \ref{prop:grad-flow-results} \eqref{item:group-invariance}, then $x \in \mathcal{B}_\delta \cap S_\delta$ also, and therefore $\mathcal{B}_\delta \subseteq S_\delta$. Since $\left\{ \mathcal{B}_\delta \right\}$ and $\left\{ S_\delta \right\}$ are both stratifications of $\mathcal{B}$, then we have $\mathcal{B}_\delta = S_\delta$ for all $\delta$.
\end{proof}

\begin{remark}  While we have identified the stable strata of the critical sets with the 
Harder-Narasimhan strata, the ordering on the set $\Delta_{\tau,d}$ coming from the values
of $\YMH$ is more complicated.  Since this will not affect the calculations, 
we continue to use the ordering already defined in Section 2.
\end{remark}

We may now reformulate the main result, Theorem \ref{thm:perfection}.  
The key idea is to define a substratification of $\left\{ X_\delta, X_\delta'
\right\}_{\delta\in \Delta_{\tau,d}}$
 by combining  $\B_\delta$ and
 $\A_{j(\delta)}$ for $\delta\in \Delta^+_{\tau,d}\cap I_{\tau,d}$.
In other words, this is simply $\{X_\delta\}_{\delta\in \Delta_{\tau,d}}$.
We call this the \emph{modified} Morse stratification.

\begin{theorem} \label{thm:perfect}
The modified Morse stratification
 $\{X_\delta\}_{\delta\in \Delta_{\tau, d}}$ is
 \emph{$\G$-equivariantly perfect} in the following sense:  For all 
$\delta\in \Delta_{\tau,d}$,  the long exact sequence

\begin{equation}
\cdots\lra H^\ast_\G(X_\delta, X_{\delta_1})\lra  H^\ast_\G(X_\delta)\lra H^\ast_\G(X_{\delta_1})\lra\cdots  \label{eqn:prime}
\end{equation}
splits.  Here, $\delta_1$ denotes the predecessor of $\delta $ in $\Delta_{\tau,d}$.
\end{theorem}

\subsection{Negative normal spaces}  \label{sec:negative}

For critical points $(A,\Phi)\in \eta_\delta$, a tangent vector 
$$(a, \varphi) \in \Omega^{0,1}(\End E ) \oplus \Omega^0(E)$$
 is an eigenvector for the Hessian of $f_\tau$ if
\begin{align}
i [\mu + i \tau \cdot \id, a] & = \lambda a \label{eqn:aeigenvalue} \\ 
i (\mu + i \tau \cdot \id) \varphi & = \lambda \varphi  \label{eqn:phieigenvalue}
\end{align}
Let  $V_{(A,\Phi)}^{neg.}\subset\Omega^{0,1}(\End E ) \oplus \Omega^0(E)$ 
denote the span of all such $(a,\varphi)$ with $\lambda<0$. 
This is clearly $\G$-invariant, since $f_\tau$ is.
 Let $\Scal_{(A,\Phi)}$ be the slice at $(A,\Phi)$.  Then we set
 $\nu_\delta\cap \Scal_{(A,\Phi)}= V_{(A,\Phi)}^{neg.}\cap \Scal_{(A,\Phi)}$.
Using Proposition \ref{prop:slice}, this gives a well-defined 
$\G$-invariant subset $\nu_\delta\subset\B$, which we call
the \emph{negative normal space at $\eta_\delta$}.  
By definition, $\eta_\delta$ is a closed subset of $\nu_\delta$.

We next describe $\nu_\delta$ 
in detail for each of the critical sets:
\begin{itemize}
\item[($\Ia$)]
  Recall that in this case $\Phi\equiv 0$.  If $E$ semistable, the negative 
eigenspace of the Hessian is $H^0(E)$. To see this, note that since $\Phi 
= 0$ then $i(\mu + i \tau \cdot \id) = \left(d/2 - \tau \right) \cdot \id$ is a 
negative constant multiple of the identity (by assumption $\tau > d/2$). Therefore $i 
[\mu + i \tau \cdot \id, a] = 0$, and $a = 0$. Then the slice equations imply $\varphi \in 
H^0(E)$.
If $E=L_1\oplus L_2$, then ${\mathcal H}^2(\Ccal_{(A,0)})$ is nonzero in general.
From the slice equations, we see that the negative eigendirections $\nu_\delta$ of the Hessian are given by
\begin{equation}\label{eqn:B2}
d_{A_2}'' \varphi_2 + a_{21} \varphi_1 = 0\, ,\ (a_{21}, \varphi_1)\in H^{0,1}(L_1^* L_2) \oplus H^0(L_1)  
\end{equation}  
\item[($\Ib$)]
This is similar to the case above, except now for negative directions, $\varphi_1\equiv 0$.    We therefore conclude that $\nu_\delta$ is given by
\begin{equation}\label{eqn:B1}
H^{0,1}(L_1^* L_2) \oplus H^0(L_2)  
\end{equation}
Note that if $\delta>\tau$, then 
$\deg L_2=d-j(\delta)<0$, and so $\nu_\delta^-$ has constant dimension
$\dim_{\C}H^{0,1}(L_1^* L_2)=2j(\delta)-d+g-1$.
\item[($\IIplus$)] In this case, $\Phi\not\equiv 0$,
 so by Lemma \ref{lem:h2},  ${\mathcal H}^2(\Ccal_{(A,0)})=0$, and the slice is 
homeomorphic to ${\mathcal H}^1(\Ccal_{(A,0)})$ via the Kuranishi map.
 The negative eigenspace of the Hessian is then just
\begin{equation}\label{eqn:C1}
 (d_A'')^* a_{21} -\varphi_2 \Phi_1^* = 0 \, ,\ d_A'' \varphi_2 + a_{21} \Phi_1 = 0  
\end{equation} 
\item[($\IIminus$)] 
 This is similar to above, except now $\varphi_2\equiv 0$.  
Hence, the fiber of $\nu_\delta$ is given by
\begin{equation}\label{eqn:C2}
H^{0,1}(L_2^* L_1)
\end{equation}
Note that $\dim_{\C}H^{0,1}(L_2^* L_1)=2j(\delta)-d+g-1$.
\end{itemize}
To see ($\IIplus$) and ($\IIminus$), we need to compute the solutions to 
\eqref{eqn:aeigenvalue} and \eqref{eqn:phieigenvalue}, which involves knowing 
the value of $i(\mu + i\tau \cdot \id)$ on the critical set. 
Equation \eqref{eqn:a-critical} shows that 
\begin{equation*}
i (\mu + i \tau \cdot \id) = \left( \begin{matrix} \lambda_1 & 0 
\\ 0 & \lambda_2 \end{matrix} \right)
\end{equation*}
where $\lambda_1 \in \Omega^0(L_1^* L_1)$ and $\lambda_2 \in \Omega^0(L_2^* L_2)$ 
are constant. Since $\Phi \in H^0(L_1) \setminus \{0\}$, then \eqref{eqn:phi-critical} 
shows that $\lambda_1 = 0$. Since $\lambda_2$ is constant,  
the integral over $M$ becomes
$$
 \lambda_2 = \frac{1}{2\pi} \int_M \lambda_2 \, dvol
 = \frac{i}{2\pi} \int_M F_{A_2} - \frac{1}{2\pi} \int_M  \tau \, dvol
  = \deg L_2 -  \tau 
$$
Therefore, if $d - \tau  < \deg L_1 = d- \deg L_2$, then $\deg L_2 < \tau $ and so $\lambda_2$ is negative. Similarly, if $\deg L_1 < d - \tau $ then $\lambda_2$ is positive. 
Equation \eqref{eqn:aeigenvalue} then shows that $a \in \Omega^{0,1}(L_1^* L_2)$ if $d - \tau  < \deg L_1$, and $a \in \Omega^{0,1}(L_2^* L_1)$ if $\deg L_1 < d - \tau $. Similarly, if $d - \tau < \deg L_1$ then $\varphi \in \Omega^0(L_2)$, and if $\deg L_1 < d - \tau $ then $\varphi = 0$.
Equations \eqref{eqn:C1} and \eqref{eqn:C2} then follow from the slice equations.

The following lemma describes the space of solutions to \eqref{eqn:B2} when $\varphi_1$ is fixed.

\begin{lemma}\label{lem:B2-deformation}
Fix $\varphi_1$. When $\varphi_1 = 0$ then the space of solutions $\{(a_{21}, \varphi_2)\}$ to \eqref{eqn:B2} is isomorphic to $H^{0,1}(L_1^* L_2) \oplus H^0(L_2)$.
When $\varphi_1 \neq 0$ then the space of solutions $\{(a_{21}, \varphi_2)\}$ to \eqref{eqn:B2} has dimension $\deg L_1$.
\end{lemma}

\begin{proof}
The first case (when $\varphi_1 = 0$) is easy, since the equations for $a \in \Omega^{0,1}(L_1^* L_2)$  and $\varphi_2 \in \Omega^0(L_2)$ become 
\begin{equation}
d_A''^* a = 0, \quad d_A'' \varphi_2 = 0 .
\end{equation} 

In the second case (when $\varphi_1 \neq 0$ is fixed), note \eqref{eqn:B2} implies that  ${\mathcal H}(a\varphi_1)=0$, where $\mathcal H$ denotes the harmonic projection $\Omega^{0,1}(L_2)\to H^{0,1}(L_2)$.  Hence, it suffices to show that the map 
\begin{equation} \label{eqn:map}
H^{0,1}(L_1^\ast L_2)\to H^{0,1}(L_2)
\end{equation}
given by multiplication with $\varphi_1$ (followed by harmonic projection) is surjective.  For then, since $\deg L_1^* L_2 < 0$, we have by Riemann-Roch that the dimension of 
\eqref{eqn:B2} is $h^0(L_2)+h^1(L_1^\ast L_2)-h^1(L_2)= \deg L_1$.
By Serre duality, \eqref{eqn:map} is surjective if and only if $H^0(KL_2^\ast)\to H^0(KL_2^\ast L_1)$ is injective.  But since $\varphi_1\neq 0$, multiplication gives an injection of sheaves ${\mathcal O}\hookrightarrow L_1$, and the result follows by tensoring and taking cohomology.
\end{proof}

\begin{lemma}\label{lem:C1-deformation}
The space of solutions to  \eqref{eqn:C1}
has constant dimension $=\deg L_1=j(\delta)$.
\end{lemma}

\begin{proof}
Consider the subcomplex $\mathcal{C}^{LT}_{(A, \Phi)}$
\begin{equation}\label{eqn:def-complex-phi-nonzero}
\xymatrix{
\Omega^0(L_1^* L_2 )) \ar[r]^(0.35){D_1} & \Omega^{0,1}(L_1^* L_2) \oplus \Omega^0(L_2) \ar[r]^(0.65){D_2} & \Omega^1(L_2)
} 
\end{equation}
Since $\Phi \neq 0$, by Lemma \ref{lem:h2}
 the cohomology at the ends of the complex \eqref{eqn:def-complex-phi-nonzero} vanishes, and we have (by Riemann-Roch)
\begin{align*}
\dim_\C \mathcal{H}^1(\mathcal{C}^{LT}_{(A, \Phi)}) & = \dim_\C( \ker D_1^* \cap \ker D_2) \\
 & = h^1(L_1^* L_2) + h^0(L_2) - h^1(L_2) - h^0(L_1^* L_2) \\
 & = - \deg L_1^* L_2 + g-1 + \deg L_2 + (1-g) \\
 & = \deg L_1
\end{align*}
\end{proof}

We summarize the the above considerations with
\begin{corollary} \label{cor:smooth}
The fiber of $\nu_\delta$ is linear of constant dimension
for critical sets of type ${\bf II}^\pm$, and for those of type $\Ib$ {\bf provided}
  $\delta\not\in \Delta_{\tau,d}^+\cap [\tau-d/2,\tau]$.  
The complex dimension of the fiber in these cases is $\sigma(\delta)$, where 
 $$ 
\sigma(\delta)=\begin{cases}
2j(\delta)-d+g-1 &\text{ if type $\Ib$ or $\IIminus$} \\
j(\delta) &\text{ if type $\IIplus$}
\end{cases}
$$
\end{corollary}

\begin{remark} \label{rem:ab}
The strata for $\delta\in I_{\tau,d}$ have
 two components corresponding to the strata
 $\A_{j(\delta)}$ and $\Bcal_\delta$.  
When there is a possible ambiguity, we will distinguish these by the notation 
 $\nu_{I,\delta}$ for the negative normal spaces to strata of type $\Ia$ or $\Ib$,
 and $\nu_{II,\delta}$ for the negative normal spaces to strata of type $\IIplus$ or $\IIminus$. 
\end{remark}

\subsection{Cohomology of the negative normal spaces}
As in \cite{DWWW}, at certain critical sets -- namely, those of type $\Ia$, $\Ib$ where $\delta\in \Delta_{\tau,d}^+\cap [\tau-d/2,\tau]$ -- the negative normal directions are not necessarily constant in dimension.  In the present case, they are not even linear.  In order to carry out the computations, we appeal to a relative sequence by considering special subspaces with better behavior.   

\begin{definition}
For $\delta\in \Delta_{\tau,d}^+\cap (\tau-d/2,\tau]$, 
let $\nu_{I,\delta}$ 
be the negative normal space to a critical set with $\Phi\equiv 0$, as in Section \ref{sec:negative}.  Define
\begin{align*}
\nu'_{I,\delta}&=\left\{ (a,\varphi_1,\varphi_2)\in \nu_{I,\delta} : (a,\varphi_1,\varphi_2)\neq 0\right\} \\
\nu''_{I,\delta}&=\left\{ (a,\varphi_1,\varphi_2)\in \nu_{I,\delta} : a\neq 0\right\} 
\end{align*}
\end{definition}
The goal of this section is the proof of the following
\begin{proposition}   \label{prop:cohomology}
\begin{align}
\delta\in \Delta_{\tau,d}^+\cap (\tau-d/2, \tau]\ : &\
H^\ast_\G(\nu_{I,\delta}, \nu''_{I,\delta})\simeq
H^{\ast-2(2j(\delta)-d+g-1)}_{S^1\times S^1}(\eta^\A_{j(\delta)}) 
 \label{eqn:nudoubleprime} \\
\delta\in \Delta_{\tau,d}^+\cap (2\tau-d, \tau]\ : &\
H^\ast_\G(\nu'_{I,\delta}, \nu''_{I,\delta})\simeq 
H^{\ast-2(2j(\delta)-d+g-1)}_{S^1}(S^{d-j(\delta)}M\times \Jac_{j(\delta)}(M)) \label{eqn:nuprime1} \\ 
\delta\in \Delta_{\tau,d}^+\cap (\tau-d/2, 2\tau-d)\ : &\
H^\ast_\G(\nu'_{I,\delta}, \nu''_{I,\delta})\simeq 
 H^{\ast-2j(\delta)}_{S^1}(S^{j(\delta)}M\times\Jac_{d-j(\delta)}(M))\label{eqn:nuprime2} \\
&\qquad \quad \oplus  H^{\ast-2(2j(\delta)-d+g-1)}_{S^1}(S^{d-j(\delta)}M\times \Jac_{j(\delta)}(M)) 
  \notag
\end{align}
\end{proposition}

\begin{proof}
Fix  $E = L_1 \oplus L_2$. Consider first the case $\tau>\deg L_1 = j(\delta) > d/2$, and $\deg L_2 = d-j(\delta) < d/2$.
Define the following spaces
\begin{align*}
\omega_\delta & = \left\{ (A_1, A_2, a, \varphi_1, \varphi_2) \in \nu_{I,\delta} \, : \, (a, \varphi_2)\neq 0 \right\}  \\
Z_\delta^- & = \left\{ (A_1, A_2, a, \varphi_1, \varphi_2) \in \nu_{I,\delta} \, : \, \varphi_1 = 0 \right\} \\
Z_\delta' & = \left\{ (A_1, A_2, a, \varphi_1, \varphi_2) \in \nu_{I,\delta} \, : \, \varphi_1 = 0, (a, \varphi_2) \neq 0 \right\} \\
Y_\delta' & = \left\{ (A_1, A_2, a, \varphi_1, \varphi_2) \in \nu_{I,\delta} \, : \, \varphi_1 \neq 0 \right\} \\
Y_\delta'' & = \left\{ (A_1, A_2, a, \varphi_1, \varphi_2) \in \nu_{I,\delta} \, : \, \varphi_1 \neq 0, (a, \varphi_2) \neq 0 \right\} \\
T_\delta &= \left\{ (A_1, A_2, a, \varphi_1, \varphi_2) \in \nu_{I,\delta}\, : \, \varphi_1 \neq 0, (a, \varphi_2) = 0 \right\} 
\end{align*}
\noindent 
Note that $Y_\delta' = \nu_{I,\delta} \setminus Z_\delta^- = \nu'_{I,\delta} \setminus 
Z_\delta'$ and $Y_\delta'' = \omega_{\delta} \setminus Z_\delta'$. Consider the following commutative diagram.

\begin{equation}\label{eqn:type-Ia-diagram}
\xymatrix{
&& \vdots \ar[d] &&&\\
&\cdots \ar[r] & H_\mathcal{G}^p(\nu_{I,\delta}, \nu'_{I,\delta}) \ar[r] \ar[d] & H_\mathcal{G}^p(\nu_{I,\delta}) \ar[r] & H_\mathcal{G}^p(\nu'_{I,\delta}) \ar[r] & \cdots \\
 && H_\mathcal{G}^p(\nu_{I,\delta}, \nu''_{I,\delta}) \ar[d] \ar[ur]^\xi \ar[dr]^{\xi''} &&&\\
\dots\ar[r] & H^p_\G(\nu'_{I,\delta},\omega_\delta) \ar[r] & H_\mathcal{G}^p(\nu'_{I,\delta}, \nu''_{I,\delta}) \ar[d] \ar[r]^\beta & H^p_\G(\omega_\delta, 
\nu''_{I,\delta}) \ar[r] & \cdots &
\\
 && \vdots &&&
}
\end{equation}

\begin{itemize}

\item   First, it follows as in the proof of \cite[Thm.\ 2.3]{DWWW} that 
the pair $(\nu_{I,\delta}, \nu''_{I,\delta})$ is homotopic to the Atiyah-Bott pair $(X^\A_{j(\delta)}, X^\A_{j(\delta)-1})$. Hence, 
\eqref{eqn:nudoubleprime} follows
from \cite{AtiyahBott83}. 

\item  Consider the pair 
$(\nu'_{I,\delta},\omega_\delta)$.  Excision of $Z_\delta'$ gives the isomorphism
\begin{equation}
H_\mathcal{G}^* (\nu'_{I,\delta}, 
\omega_\delta) \cong H_\mathcal{G}^*(\nu'_{I,\delta} \setminus Z_\delta', \omega_\delta \setminus Z_\delta') \cong H_\mathcal{G}^*(Y_\delta', Y_\delta'') 
\end{equation}
The space $Y_\delta'' = Y_\delta'
 \setminus T_\delta$, and Lemma \ref{lem:B2-deformation} shows that 
$Y_\delta'$ is a vector bundle over $T_\delta$ with fibre dimension $=\deg L_1$. Therefore the Thom isomorphism implies
$$
H_\mathcal{G}^*(Y_\delta', Y_\delta'') = H_\mathcal{G}^*(Y_\delta', Y_\delta' \setminus X_\delta') \cong H_\mathcal{G}^{*-2j(\delta)}(T_\delta)
$$
and therefore 
\begin{equation}
H_\G^\ast(\nu'_{I,\delta}, \omega_\delta) = H_\G^\ast(Y_\delta', Y_\delta'') = H^{\ast-2j(\delta)}_{S^1}(S^{j(\delta)} M \times \Jac_{d-j(\delta)}(M)) \label{eqn:omegaprime}
\end{equation}

\item  Consider $(\omega_\delta, \nu''_{I,\delta})$.
  By retraction, the pair  is homotopic to the intersection with $\varphi_1=0$.  It then follows exactly as in \cite{DWWW} (or the argument above) that
\begin{equation}
H_\mathcal{G}^*(\omega_\delta, \nu''_{I,\delta})\cong H_{S^1}^{*-2(2j(\delta)-d+g-1)}(S^{d-j(\delta)} M\times\Jac_{j(\delta)}(M))
\label{eqn:omegadoubleprime}
\end{equation}
(Recall that $\dim H^{0,1}(L_1^* L_2) = 2j(\delta)-d+g-1$ by Riemann-Roch, and that $\deg L_2 = d-j(\delta)$).
\end{itemize}
It then follows as in \cite{DWWW} that $\xi''$, and hence also $\beta$,  is surjective.  This implies that  the lower horizontal exact sequence splits, and
\eqref{eqn:nuprime2} follows from \eqref{eqn:omegaprime} and \eqref{eqn:omegadoubleprime}. 
This completes the proof in this case.
The case where $\deg L_1>\tau$ is simpler, since $\varphi_1\equiv 0$ from \eqref{eqn:phieigenvalue}.  Hence, $\omega_\delta=\nu_{I,\delta}'$, and the proof proceeds as above.
\end{proof}

\subsection{The Morse-Bott lemma} \label{sec:bott}
In this section we prove the fundamental relationship between the relative cohomology of successive strata and  the relative cohomology of the negative normal spaces.  From this we derive the proof of the main result. In the following we use $\delta_1$ to denote the predecessor of $\delta$ in $\Delta_{\tau, d}$.

\begin{theorem}  \label{thm:bott}
For all $\delta\in \Delta_{\tau,d}\setminus \left( \Delta_{\tau,d}^+\cap I_{\tau,d} \right)$,
\begin{equation} \label{eqn:bott1}
H^\ast_\G(X_{\delta}, X_{\delta_1})\simeq H^\ast_\G(\nu_\delta, \nu'_\delta)
\end{equation}
For all $\delta\in \Delta_{\tau,d}^+\cap I_{\tau,d}$,
\begin{equation} \label{eqn:bottprime}
H^\ast_\G(X_{\delta}', X_{\delta_1})\simeq H^\ast_\G(\nu_{II,\delta}, \nu'_{II,\delta})
\end{equation}
For all $\delta\in \Delta_{\tau,d}^+\cap I_{\tau,d}$, $\delta\neq \tau-d/2$,
\begin{equation} \label{eqn:bott}
H^\ast_\G(X_{\delta}, X_{\delta}')\simeq H^\ast_\G(\nu_{I,\delta}, \nu'_{I,\delta})
\end{equation}
 Eq.\ \eqref{eqn:bott}  also holds for  $\delta=\tau-d/2$, provided $d>4g-4$.
In the statements above, $\delta_1$ denotes the predecessor of $\delta$ in  $\Delta_{\tau,d}$.
\end{theorem}

First, we give a proof of \eqref{eqn:bott1}
in the case $\delta\not\in\Delta_{\tau,d}^+\cap[\tau-d/2,\tau]$. 
  By excision and  convergence of the gradient flow, there is a neighborhood $U$ of $\eta_\delta$ such that
\begin{itemize}
\item $U$ is $\G$-invariant;
\item $U$ is the union of images of slices $\Scal_{(A,\Phi)}$, where $(A,\Phi)\in \eta_\delta$;
\item  $H^\ast_\G(X_{\delta}, X_{\delta_1})\simeq H^\ast_\G(U,U\setminus (U\cap \B_\delta))$
\end{itemize}
Notice that for each slice 
$\Scal_{(A,\Phi)}\cap U\setminus (U\cap \B_\delta)=\Scal_{(A,\Phi)}'\cap U$,
where the latter is defined as in Lemma \ref{lem:retraction1}.  By the lemma, it follows that the pair
$(U,U\setminus (U\cap \B_\delta))$ locally retracts to 
$(\nu_\delta, \nu_\delta')$.  On the other hand, by Corollary \ref{cor:smooth}, $\nu_\delta$ 
is a bundle over $\eta_\delta$.  It follows by continuity as in \cite{Bott54}, that there is 
a $\G$-equivariant retraction of the pair $(\nu_\delta, \nu_\delta')
\hookrightarrow(U,U\setminus (U\cap \B_\delta))$.  The result therefore follows in this case. 
We also note that by Corollary \ref{cor:smooth} and the Thom isomorphism,
\begin{equation} \label{eqn:thom_eta}
H^\ast_\G(\nu_\delta, \nu_\delta')\simeq H^{\ast-2\sigma(\delta)}_\G(\eta_\delta)
\end{equation}

\begin{remark} \label{rem:bott}
 Notice that by Corollary \ref{cor:smooth} the same argument also proves \eqref{eqn:bottprime}.  For $d>4g-4$, we can use Lemma \ref{lem:retraction2} in the same way to derive
 \eqref{eqn:bott} for $\delta=\tau-d/2$.  In this case, by the Thom isomorphism, we have
\begin{equation} \label{eqn:minimal_ym}
H^\ast_\G(X_{\tau-d/2}, X_{\tau-d/2}')\simeq H^{\ast-2(d+2-2g)}_\G(\A_{ss})
\end{equation}
 \end{remark}

\begin{lemma} \label{lem:claim}
For $\delta\not\in\Delta_{\tau,d}^+\cap[\tau-d/2,\tau]$, or if $\delta=\tau-d/2$ and $d>4g-4$, then
the long exact sequence \eqref{eqn:prime}
splits.  Similarly, the long exact sequence 
\begin{equation}
\cdots\lra H^p_\G(X_\delta', X_{\delta_1})\lra H^p_\G(X_\delta')\lra H^p_\G(X_{\delta_1})\lra\cdots \label{eqn:1prime}
\end{equation}
splits for all $\delta\in I_{\tau,d}$.
\end{lemma}

\begin{proof}
Indeed, since \eqref{eqn:bott1} holds in this case, we have 
\begin{equation}\label{eqn:diagram}
\xymatrix{
\cdots \ar[r] & H_{\mathcal{G}}^p(X_\delta, X_{\delta_1}) \ar[d]^{\cong} \ar[r]^{\alpha} & H_{\mathcal{G}}^p(X_\delta) \ar[r] \ar[d] & H_{\mathcal{G}}^p(X_{\delta_1})  \ar[r] & \cdots \\
& H_\mathcal{G}^p(\nu_\delta, \nu'_\delta) \ar[r]^{\beta} & H_\mathcal{G}^p(\eta_\delta)   &  & 
}
\end{equation}
Now $\nu_\delta\to \eta_\delta$ is a complex vector bundle with a $\G$-action and a circle subgroup that fixes $\eta_\delta$ and acts freely on $\nu_\delta \setminus \eta_\delta$, so by \cite[Prop.\ 13.4]{AtiyahBott83}, $\beta$ is injective.  It follows that
is $\alpha$ is injective as well, and hence the sequence splits.  The second statement follows by Remark \ref{rem:bott} and the same argument as above.
\end{proof}

It remains to prove \eqref{eqn:bott} and the remaining cases of \eqref{eqn:bott1}.  As noted above, 
in these cases the negative normal spaces are no longer constant in dimension, 
and indeed they are not even linear in the fibers.  From the point of 
view of deformation theory, the Kuranishi map near these critical sets is not surjective, and 
defining an appropriate retraction is more difficult than in the situation just 
considered.  Instead, we resort to the analog of the decomposition used in Section 
\ref{sec:negative}.
Let
$
X_{\delta}^{\prime\prime}=X_{\delta}\setminus \pr^{-1}(\A_{j(\delta)})
$.
Note that by Lemma \ref{lem:projection},
$X_\delta^{\prime\prime}\subset X_{\delta}'$.
We will prove the following

\begin{proposition}  \label{prop:key}
Suppose $\delta\in \Delta^+_{\tau,d}\cap (\tau-d/2,\tau]$.  Then
\begin{align}
H^\ast_\G(X_{\delta}, X_\delta^{\prime\prime})&\cong H^\ast_\G(\nu_{I,\delta}, \nu''_{I,\delta})) \label{eqn:a_prime} \\
H^\ast_\G(X_{\delta}', X_\delta^{\prime\prime})&\cong H^\ast_\G(\nu'_{I,\delta}, \nu''_{I,\delta}) \label{eqn:b_prime}
\end{align}
\end{proposition}

\begin{proof}[Proof of \eqref{eqn:a_prime}]
By \cite{AtiyahBott83} and \eqref{eqn:nudoubleprime}, it suffices to prove
\begin{equation} \label{eqn:ab}
H^\ast_\G(X_{\delta}, X_\delta^{\prime\prime})\cong H^\ast_\G(X^\A_{j(\delta)}, X^\A_{j(\delta)-1})
\end{equation}
We first note that the pair $(X_{\delta}, X_\delta^{\prime\prime})$ is not necessarily invariant under scaling $t\Phi$, $t\to 0$, in particular because of the strata in $\Delta_{\tau,d}^-$ (cf.\ Lemma \ref{lem:projection}).   However, if we set
$$
\widehat X_{\delta}= X_{\delta}\cup\bigcup_{\delta'\leq \delta\, ,\, \delta'\in \Delta_{\tau,d}^-} X^\A_{\delta'+\tau}
 \ ,\
\widehat X_{\delta}^{\prime\prime}=\widehat X_{\delta}\setminus \pr^{-1}(\A_{j(\delta)})$$
then by excision on the closed  subset 
$$
\bigcup_{j(\delta)-\tau<\delta'\leq \delta\atop \delta'\in \Delta_{\tau,d}^-} \A_{\delta'+\tau} 
$$
it follows that 
$
H^\ast_\G(X_{\delta}, X_\delta^{\prime\prime})=H^\ast_\G(\widehat X_{\delta}, \widehat X_\delta^{\prime\prime})
$.
Then for the pair $(\widehat X_{\delta}, \widehat X_\delta^{\prime\prime})$, projection to $\A$ is a deformation retraction (by scaling the section $\Phi$), and we have
\begin{equation} \label{eqn:xhat}
H^\ast_\G(X_{\delta}, X_\delta^{\prime\prime})= H^\ast_\G(\pr(\widehat X_{\delta}), \pr(\widehat X_\delta^{\prime\prime}))
\end{equation}
Next, let
$$
\K_\delta=\pr\bigl(\widehat X_{(\tau-d/2)}\cup \bigcup _{\delta'\leq \delta\, ,\, \delta'\in \Delta_{\tau,d}^-} \Bcal_{\delta'}  
\cup \bigcup _{\delta'< \tau-d/2\, ,\, \delta'\in \Delta_{\tau,d}^+} \Bcal_{\delta'}   \bigr) \cap \bigcup_{k>j(\delta)} \A_{k}
$$
Note that $\K_\delta\subset \pr(\widehat X_\delta^{\prime\prime})$.  We claim that it is actually a closed subset of $\pr(\widehat X_{\delta})$.  Indeed, suppose $(A_i, \Phi_i)\in X_{(\tau-d/2)}$, $(A_i, \Phi_i)\to (A, \Phi)\in \widehat X_{\delta}$, and suppose that $\mu_+(A_i)>j(\delta)$ for each $i$.  By semicontinuity, it follows that $\mu_+(A)>j(\delta)$.  On the other had, either $A\in\K_\delta$ or 
$(A,\Phi)\in \Bcal_{\delta'}$, $\tau-d/2<\delta'\leq \delta$ and $\delta'\in \Delta_{\tau,d}^+$.  But by Lemma \ref{lem:projection}, 
this would imply $A\in \A_{j(\delta')}$; which is a contradiction, since $j(\delta')\leq j(\delta)$.
 It follows that the latter cannot occur, and hence,  $\K_\delta$ is closed.
Similarly, 
\begin{align*}
\pr(\widehat X_{\delta})&=\pr\bigl(\widehat X_{(\tau-d/2)}\cup \bigcup _{\delta'\leq \delta\, ,\, \delta'\in \Delta_{\tau,d}^-} \Bcal_{\delta'}  
\cup \bigcup _{\delta'< \tau-d/2\, ,\, \delta'\in \Delta_{\tau,d}^+} \Bcal_{\delta'}   \bigr) \\
&\qquad\qquad\cup \bigcup_{d/2<k\leq j(\delta)} \A_{k}\cup \bigcup_{\tau-d/2<\delta'\leq \delta\, ,\, \delta'\in \Delta_{\tau,d}^+} \pr(\Bcal_\delta) \\
&=\K_\delta\cup   \bigcup_{d/2\leq k\leq j(\delta)} \A_k
\end{align*}
and the union is disjoint. It follows also that
$$
\pr(\widehat X_\delta'')=\K_\delta\cup   \bigcup_{d/2\leq k< j(\delta)} \A_k
$$
Hence, $\pr(\widehat X_\delta)\setminus \K_\delta=X_{j(\delta)}^\A$, $\pr(\widehat X_{j(\delta)}'')\setminus \K_\delta=X_{j(\delta)-1}^\A$,
and \eqref{eqn:ab} follows from \eqref{eqn:xhat} by excision.
\end{proof}

\begin{proof}[Proof of \eqref{eqn:b_prime}]
First consider the case $\delta\in \Delta_{\tau,d}^+ \cap (\tau-d/2, 2\tau-d)$.
We have
\begin{align*}
X_\delta''&=\bigl( X_{(\tau-d/2)}\cup \bigcup _{\delta'\leq \delta\, ,\, \delta'\in \Delta_{\tau,d}^-} \Bcal_{\delta'}  
\cup \bigcup _{\delta'< \tau-d/2\, ,\, \delta'\in \Delta_{\tau,d}^+} \Bcal_{\delta'}   \bigr)\setminus \pr^{-1}(\A_{j(\delta)})
 \\
&\qquad\qquad\cup \bigcup_{d/2<k< j(\delta)} \A_{k}\cup \bigcup_{\tau-d/2<\delta'< \delta\, ,\, \delta'\in \Delta_{\tau,d}^+} \Bcal_\delta 
\end{align*}
whereas 
$
X_\delta'=X_{\delta_1}\cup \Bcal_\delta
$, 
where $\delta_1$ is the predecessor of $\delta $ in $\Delta_{\tau,d}$.
Also,  $X_\delta''=X_{\delta_1}\setminus \pr^{-1}(\A_{j(\delta)}) $.  We then have the following diagram
\begin{equation} \label{eqn:trick}
\xymatrix{
\cdots \ar[r]  &\ar[d]^{f} H^p_\G(X_\delta', X_\delta'') \ar[r] &\ar[d]^{g} H^p_\G(X_\delta') \ar[r]  & \ar[d]^{\cong} H^p_\G(X_\delta'') \ar[r] & \cdots \\
\cdots \ar[r] &H^p_\G(X_{\delta_1}, X_\delta'') \ar[r] & H^p_\G(X_{\delta_1}) \ar[r]  & H^p_\G(X_\delta'') \ar[r] &\cdots
}
\end{equation}
where $f$ and $g$ are induced by the inclusion $X_{\delta_1}\hookrightarrow X_\delta'$.  By  Lemma \ref{lem:claim} and \eqref{eqn:bottprime} (see Remark \ref{rem:bott}), it follows that $g$ is surjective and
$$\ker g = H^\ast_\G(\nu_{II,\delta}, \nu'_{II,\delta})\simeq H^{\ast-2j(\delta)}_\G(\B_{\delta})\simeq H^{\ast-2j(\delta)}_{S^1}(S^{j(\delta)}M\times\Jac_{d-j(\delta)} M)$$
by Thom isomorphism.  Chasing through the diagram, it follows that $f$ is also surjective with the same kernel.  We conclude that
\begin{equation} \label{eqn:step_one}
 H^\ast_\G(X_\delta', X_\delta'')\simeq H^\ast_\G(X_{\delta_1}, X_\delta'')\oplus H^{\ast-2j(\delta)}_{S^1}(S^{j(\delta)}M\times\Jac_{d-j(\delta)} M)
\end{equation} 
It remains to compute the first factor on the right hand side.  To begin, notice that 
$$
\bigcup_{d/2<k< j(\delta)} \A_{k}\cup \bigcup _{\tau-d/2<\delta'< \delta\, ,\, \delta'\in \Delta_{\tau,d}^-} \Bcal_{\delta'}\cup  \bigcup_{\tau-d/2<\delta'< \delta\, ,\, \delta'\in \Delta_{\tau,d}^+} \Bcal_\delta 
$$
is contained in $X_\delta''$ and closed in $X_{\delta_1}$.  It follows by excision that
$$
H^\ast_\G(X_{\delta_1}, X_\delta'')\simeq H^\ast_\G(X_{\tau-d/2}, X_{\tau- d/2}\setminus \pr^{-1}(\A_{j(\delta)}))
$$
Next, we observe that 
$$
\A_{ss}\cup \bigcup _{\delta'< \tau-d/2\, ,\, \delta'\in \Delta_{\tau,d}^-} \Bcal_{\delta'}  
$$
 is contained in $X_{\tau-d/2}\setminus \pr^{-1}(\A_{j(\delta)})$ and closed in $X_{\tau-d/2}$.  This is clear for $\A_{ss}$.  More generally, 
 if $(E,\Phi)$ in this set and $\Phi\not\equiv 0$, then $\mu_+(E)>\tau>j(\delta)$, and elements in the strata of type $\IIminus$ cannot specialize to points in $\IIplus$.
 Again applying  excision, we have
$$
H^\ast_\G(X_{\delta_1}, X_\delta'')\simeq H^\ast_\G(Y_\delta , Y_\delta\setminus \pr^{-1}(\A_{j(\delta)})) \\
$$
where
$$
Y_\delta= \B^{\tau}_{ss}\cup  \bigcup_{0<\delta'\leq \tau-d/2\, ,\, \delta'\in \Delta_{\tau,d}^+} \Bcal_{\delta'} 
$$
We make a third excision of the closed set 
$$
\bigcup_{\tau-j(\delta)<\delta'\leq \tau-d/2\, ,\, \delta'\in \Delta_{\tau,d}^+} \Bcal_{\delta'} 
$$
and a final excision of the subset 
$$
{\mathcal D}_\delta=\bigl\{\B^{\tau}_{ss}\cup  \bigcup_{0<\delta'\leq \tau-j(\delta)\, ,\, \delta'\in \Delta_{\tau,d}^+} \Bcal_\delta\bigr\}   \cap \bigl( \bigcup_{k>j(\delta)}\pr^{-1}(\A_k)\bigr)
$$
Notice that 
$$
\bigl\{\B^{\tau}_{ss}\cup  \bigcup_{0<\delta'\leq \tau-j(\delta)\, ,\, \delta'\in \Delta_{\tau,d}^+} \Bcal_\delta\bigr\} \setminus {\mathcal D}_\delta= \B^{j(\delta)}_{ss}
$$
We conclude that
$$
H^\ast_\G(X_{\delta_1}, X_\delta'')\simeq H^\ast_\G(\B^{j(\delta)}_{ss}, \B^{j(\delta)}_{ss}\setminus\pr^{-1}(\A_{j(\delta)}))
$$
Choose $\varepsilon>0$ small, and let $\tau'=j(\delta)-\varepsilon$.  
Then with respect to the \emph{$\tau'$-stratification}, the right hand side above is
$
\simeq H^\ast_\G(\B^{\tau'}_{ss}\cup \Bcal_\varepsilon^{\tau'},   \, \B^{\tau'}_{ss})
$
where $\varepsilon\in \Delta_{\tau'}^-$ is the lowest $\tau'$-critical set.  Since $\varepsilon<\tau'-d/2$, it follows from Lemma \ref{lem:claim} that the long exact sequence
 \eqref{eqn:prime} splits for this stratum.
Hence, we have
\begin{equation} \label{eqn:step_two}
H^\ast_\G(X_{\delta_1}, X_\delta'')\simeq H^\ast_\G(\B^{\tau'}_{ss}\cup \Bcal_\varepsilon^{\tau'},   \, \B^{\tau'}_{ss})\simeq H^{\ast-2(2j(\delta)-d+g-1)}_{S^1}(S^{d-j(\delta)}M\times\Jac_{j(\delta)}(M))
\end{equation}
(notice that $j_{\tau'}(\varepsilon)=j_\tau(\delta)$).
Eqs.\ \eqref{eqn:step_one} and \eqref{eqn:step_two}, combined with Proposition \ref{prop:cohomology},
 complete the proof.   In case $\delta\not\in I_{\tau,d}$, note that by definition 
$
H^\ast_\G(X_\delta', X_\delta'')\simeq H^\ast_\G(X_{\delta_1}, X_\delta'')
$.
The part of the proof following \eqref{eqn:step_one}   now applies verbatim to this case.
\end{proof}

\begin{proof}[Proof of Theorem \ref{thm:bott}] 
For $\delta\not\in  \Delta^+_{\tau,d}\cap[\tau-d/2,\tau]$, or $\delta=\tau-d/2$ and $d>4g-4$, we have proven the result directly (see the discussion following Theorem \ref{thm:bott} and also Remark \ref{rem:bott}). For $\delta\in \Delta^+_{\tau,d}\cap (\tau-d/2,\tau]$, the result follows from Proposition \ref{prop:key} and the five lemma. 
\end{proof}

\subsection{Perfection of the stratification for large degree} \label{sec:large_degree}
Note that Lemma \ref{lem:claim} shows that the long exact sequence
\eqref{eqn:prime} splits for all $\delta \notin \Delta_{\tau, d}^+
 \cap [\tau - d/2, \tau]$, and also for 
$\delta = \tau - d/2$ if $d > 4g-4$.
 Therefore it remains to show that \eqref{eqn:prime} splits
 for $\delta \in \Delta_{\tau,d}^+ \cap (\tau - d/2, \tau]$.

Firstly we consider the case where $\delta \in \Delta_{\tau,d}^+
 \cap [2\tau - d, \tau]$, which corresponds to a stratum of type $\Ib$. 
Proposition \ref{prop:key}  shows that the vertical long exact sequence splits 
and the map $\xi$ is injective in the following commutative diagram.

\begin{equation*}
\xymatrix{
 & \vdots \ar[d] \\
 \cdots \ar[r] & H_\mathcal{G}^p(X_\delta, X_{\delta_1}) \ar[r]^{\alpha^p} \ar[d]^{\zeta^p} & H_\mathcal{G}^p(X_\delta) \ar[r] & H_\mathcal{G}^p(X_{\delta_1}) \ar[r] & \cdots \\
 H_\mathcal{G}^p(\nu_{I, \delta}, \nu_{I, \delta}'') \ar[d] 
\ar[r]^\cong & H_\mathcal{G}^p(X_\delta, X_\delta'') \ar[d] \ar[ur]^\xi \\
H_\mathcal{G}^p(\omega_\delta, \nu_{I, \delta}'')
 \ar[r]^\cong \ar[d] & H_\mathcal{G}^p(X_{\delta_1}, X_\delta'') \ar[d] \\
 0 & \vdots 
}
\end{equation*}
Therefore the map $\alpha^p$ is injective, and so the horizontal long exact sequence splits also.

Next, suppose
 $\delta \in \Delta_{\tau, d}^+ \cap (\tau - d/2, 2\tau - d)$. 
For this we need the following lemma.

\begin{lemma}
When $\delta \in \Delta_{\tau, d}^+ \cap (\tau - d/2, 2\tau - d)$, 
then the isomorphisms $H_\mathcal{G}^* (X_\delta, X_\delta'') \cong H_\mathcal{G}^*(\nu_{I, \delta}, \nu_{I, \delta}'')$ and $H_\mathcal{G}^*(X_{\delta_1}, X_\delta'') \cong H_\mathcal{G}^*(\omega_\delta, \nu_{I, \delta}'')$ in equivariant cohomology are induced by an inclusion of triples $(\nu_{I, \delta}, \omega_\delta, \nu_{I, \delta}'') \hookrightarrow (X_\delta, X_{\delta_1}, X_\delta'')$.
\end{lemma}

\begin{proof}
The first isomorphism is contained in \eqref{eqn:a_prime}. To see 
the second isomorphism, note that the results of the last section show that 
$H_\mathcal{G}^*( X_{\delta_1}, X_\delta'') \cong H_\mathcal{G}^*( 
\mathcal{B}_{ss}^{\tau'} \cup \mathcal{B}_\varepsilon^{\tau'}, 
\mathcal{B}_{ss}^{\tau'})$, where $\varepsilon \in \Delta_{\tau'}^-$ is 
the lowest $\tau'$ critical set. Excise all but a neighborhood of 
$\mathcal{B}_\varepsilon^{\tau'}$, and deformation retract $\Phi$ so 
that $\| \Phi \|$ is small. Call these new sets $W$ and $W_0$,  respectively. Then
\begin{equation*}
H_\mathcal{G}^*(\mathcal{B}_{ss}^{\tau'} \cup \mathcal{B}_\varepsilon^{\tau'}, 
\mathcal{B}_{ss}^{\tau'}) \cong H_\mathcal{G}^*(W, W_0) 
\end{equation*}
Since $\Phi \neq 0$, then we can apply Lemma \ref{lem:retraction1} to the slices 
within the spaces $W$ and $W_0$, and the resulting spaces are 
homeomorphic to $\omega_\delta$ and $\nu_{I, \delta}''$ respectively.
\end{proof}

The previous lemma together with the 
surjection $\xi'' : H_\mathcal{G}^*( \nu_{I, \delta}^-, \nu_{I, \delta}'') 
\rightarrow H_\mathcal{G}^*(\omega_\delta, \nu_{I, \delta}'')$ from 
\eqref{eqn:type-Ia-diagram} implies that the map $\xi_g''$ is 
surjective in the following commutative diagram.

\begin{equation*}
\xymatrix{
& \vdots \ar[d] \\
\cdots \ar[r] & H_\mathcal{G}^p(X_\delta, X_{\delta_1}) \ar[r] \ar[d] & H_\mathcal{G}^p(X_\delta) \ar[r] & H_\mathcal{G}^p(X_{\delta_1}) \ar[r] & \cdots \\
 & H_\mathcal{G}^p(X_\delta, X_\delta'') \ar[d]^{\xi_g''} \ar[ur]^{\xi_g} \\
 & H_\mathcal{G}^p(X_{\delta_1}, X_\delta'') \ar[d] \\
 & \vdots
}
\end{equation*}
The isomorphism \eqref{eqn:ab} together with the results of 
\cite{AtiyahBott83} show that the map $\xi_g$ is injective, 
and so the same argument as before shows that the horizontal long exact sequence splits.

\subsection{The case of low degree} \label{sec:low_degree}
By the results of the previous section,  there is only one critical stratum unaccounted for on the way to completing the proof of Theorem \ref{thm:perfect} for $1\leq d\leq 4g-4$.  Namely,
 we need  to analyze what happens when we attach the minimal Yang-Mills stratum $\A_{ss}$, which is the lowest critical set of Type {\bf I}.  More precisely, from  \eqref{eqn:x_xprime}, we need to
 show that the inclusion $X'_{\tau-d/2}\hookrightarrow  X_{\tau-d/2}$ 
 induces a surjection in $\G$-equivariant rational cohomology  for all $\tau\in (d/2,d)$.
 Notice that by \eqref{eqn:xprime},
$ X'_{\tau-d/2}=X_{\delta_1}$ for $d$ odd, so this is precisely what we need to prove;
and if $d$ is even, then the above statement together with
Lemma \ref{lem:claim}
will prove that $X_{\delta_1}\hookrightarrow  X_{\tau-d/2}$ 
 induces a surjection in $\G$-equivariant rational cohomology 
in this case as well.
 
 In low degree, the negative normal directions exist only over a Brill-Noether subset of $\A_{ss}$, whose cohomology is unknown, and the
  dimension of the fiber jumps in a complicated way;  it is not even clear that there is a good Morse-Bott lemma of the type \eqref{eqn:bott} in this case.
 
 Hence, in order
 to prove surjectivity in this case
we will use an indirect argument via embeddings of the space of pairs of
 degree $d$ into corresponding pairs of larger degree.
  More precisely, this is defined as follows.
Choose a point $p\in M$, and let ${\mathcal O}(p)$
 denote the holomorphic line bundle with divisor $p$.   We 
also choose a hermitian metric on ${\mathcal O}(p)$. 
 Choose a holomorphic section $\sigma_p$ of ${\mathcal O}(p)$ with a 
simple zero at $p$.  Note that $\sigma_p$ is unique up to a nonzero multiple. 
 A holomorphic (and hermitian) structure on  the complex 
vector bundle $E$ induces one on the bundle $\widetilde E=E\otimes 
{\mathcal O}(p)$. Moreover, if $\Phi\in H^0(E)$, then
 $\widetilde\Phi=\Phi\otimes\sigma_p\in H^0(\widetilde E)$.  The unitary gauge group $\G$ of $E$ is 
canonically isomorphic to that of $\widetilde E$.  Hence, we have a $\G$-equivariant embedding $\B(E)\hookrightarrow 
\B(\widetilde E)$.  For simplicity, we will use the notation $\B=\B(E)$ and $\widetilde \B=\B(\widetilde E)$.

Let $\tilde d=d+2$ and $\tilde \tau=\tau+1$.  Then we note the following properties:
\begin{eqnarray*}
\deg\widetilde E=\tilde d & \Delta_{\tilde \tau,\tilde d}=\Delta_{\tau,d} \\
\deg \widetilde\Phi=\deg \Phi+1 & I_{\tilde \tau,\tilde d}= I_{\tau,d} \\
\mu_+(\widetilde E)=\mu_+(E)+1 & j_{\tilde \tau,\tilde d}(\delta)= j_{\tau,d} +1
\end{eqnarray*}
It follows easily that the  inclusion respects the Harder-Narasimhan stratification, i.e.\ for all $\delta\in 
\Delta_{\tau,d}$, $\B_\delta\hookrightarrow \widetilde \B_\delta$,  $X_\delta\hookrightarrow \widetilde X_\delta$,  and $X_\delta'\hookrightarrow 
\widetilde X_\delta'$, where the tilde's have the obvious meaning.  In particular, if we fix $\tau_{max}=d-\varepsilon$, 
for $\varepsilon $ small, then $\B^{\tau_{max}}_{ss}\hookrightarrow \widetilde \B^{\tilde \tau_{max}}_{ss}$.  Notice that while $\B^{\tau_{max}}_{ss}$ 
gives the ``last" moduli space in the sense that there are no critical values between ${\tau_{max}}$ and $d$ 
(provided $\varepsilon$ is sufficiently small),  $\widetilde \B^{\tilde\tau_{max}}_{ss}$ gives the ``second to last" moduli 
space in the sense that there is precisely one critical value between $\tilde \tau_{max}$ and $\tilde d$.

\begin{lemma} \label{lem:maxtau}
The
 inclusion $\B^{\tau_{max}}_{ss}\hookrightarrow \widetilde \B^{\tilde \tau_{max}}_{ss}$ induces a surjection in $\G$-equivariant rational cohomology.
\end{lemma}

\begin{proof}
Since $\tau$ is generic, it suffices to prove the result on the level of moduli spaces, i.e.\ that the inclusion
$\imath:\M_{\tau_{max},d}\hookrightarrow \widetilde \M_{\tilde \tau_{max},\tilde d}$ induces a surjection in  cohomology.
Consider the determinant map $(E,\Phi)\mapsto \det E$.  We have the following diagram
\begin{equation}
\xymatrix
{\M_{\tau_{max},d} \ar[r]^{\imath} \ar[d]^{\det} &\widetilde \M_{\tilde \tau_{max},\tilde d} \ar[d]^{\det}\\
\Jac_d(M) \ar[r]^{\jmath} &\Jac_{\tilde d}(M) 
}
\end{equation}
Now $\M_{\tau_{max},d}$ is  the projectivization of a vector bundle (cf.\ \cite{Thaddeus94}).  Hence, by the Leray-Hirsch theorem its cohomology ring is generated by the embedding  $(\det)^\ast(H^\ast(\Jac_d(M)))$, and a $2$-dimensional class generating the cohomology of the fiber. Since $\imath^\ast(\det)^\ast=(\det)^\ast\jmath^\ast$, and $j^\ast$ is an isomorphism, it follows that $\imath^\ast$ is surjective onto $(\det)^\ast(H^\ast(\Jac_d(M)))$.  It remains to show that the $2$-dimensional class is  in the image of $\imath^\ast$.  But since $\imath$ is holomorphic and $\M_{\tilde \tau_{max},\tilde d}$ is projective, the K\"ahler class of $\M_{\tilde \tau_{max},\tilde d}$ restricted to the image generates the cohomology of the fiber.
\end{proof}

\begin{lemma} \label{lem:delta}
Suppose $\delta\in \Delta_{{\tau_{max}},d}$, $\delta<{\tau_{max}}-d/2$.  Then the inclusion $X_\delta\hookrightarrow \widetilde X_\delta$ induces a surjection in $\G$-equivariant rational cohomology.  The same holds for $X'_{\tau-d/2}\hookrightarrow \widetilde X'_{\tilde\tau-\tilde d/2}$.
\end{lemma}

\begin{proof}
By Lemma \ref{lem:maxtau}, the result holds for the semistable stratum.  Fix $\delta<\tau-d/2$, and let $\delta_1$ be its predecessor in $ \Delta_{{\tau_{max}},d}$.  By induction, we may assume the result holds for $\delta_1$.  By Lemma \ref{lem:claim} we have the following diagram:
\begin{equation} \label{eqn:embedding}
\xymatrix{
0 \ar[r] & H_{\mathcal{G}}^p(\widetilde X_\delta, \widetilde X_{\delta_1}) \ar[d]^{f} \ar[r] & H_{\mathcal{G}}^p(\widetilde X_\delta) \ar[r] \ar[d]^{g} & H_{\mathcal{G}}^p(\widetilde X_{\delta_1}) \ar[d]^{h} \ar[r] & 0\\
0 \ar[r] & H_\mathcal{G}^p(X_\delta, X_{\delta_1}) \ar[r] & H_\mathcal{G}^p(X_\delta)\ar[r]  & H_\mathcal{G}^p(X_{\delta_1}) \ar[r] & 0 
}
\end{equation}
By the inductive hypothesis, $h$ is surjective.  On the other hand, by \eqref{eqn:bott1} and \eqref{eqn:thom_eta},
 surjectivity of $f$ is equivalent to surjectivity of the map $H^\ast_\G(\tilde\eta_\delta)\to H^\ast_\G(\eta_\delta)$.  
 From the description of critical sets (cf.\ Proposition \ref{prop:critical_set_cohomology}), this map is induced by the inclusion $S^{j(\delta)}M\hookrightarrow S^{j(\delta)+1}M$.
 Surjectivity then follows by the argument in  \cite[Sect. 4]{DWWW}.  Since both $f$ and $h$ are surjective, so is $g$. The result for any $\delta<\tau-d/2$ now follows by induction.  If $d$ is even, the exact same argument, with $\delta_1=$ the predecessor of $\tau-d/2$,  proves the statement for  $X'_{\tau-d/2}$ as well.
\end{proof}

\begin{lemma} \label{lem:induction}
  Suppose the inclusion $\widetilde X'_{\tilde\tau_{max}-\tilde d/2}\hookrightarrow \widetilde X_{\tilde\tau_{max}-\tilde d/2}$  induces a surjection in $\G$-equivariant rational cohomology.
Then the same is true for the inclusion $X'_{{\tau_{max}}-d/2}\hookrightarrow  X_{{\tau_{max}}-d/2}$.
\end{lemma}

\begin{proof}
Consider the diagram
\begin{equation}
\xymatrix{
H_{\mathcal{G}}^p(\widetilde X_{\tilde\tau_{max}-\tilde d/2}) \ar[r] \ar[d] & H_{\mathcal{G}}^p(\widetilde X'_{\tilde\tau_{max}-\tilde d/2}) \ar[d]^{h} \ar[r] & 0\\
 H_\mathcal{G}^p(X_{\tau_{max}-d/2})\ar[r]  & H_\mathcal{G}^p(X'_{\tau_{max}-d/2}) \ar[r] & \cdots
}
\end{equation}
By Lemma \ref{lem:delta}, $h$ is surjective.  The result then follows immediately.
\end{proof}

\begin{lemma} \label{lem:anytau}
 Suppose the inclusion $X'_{\tau-d/2}\hookrightarrow  X_{\tau-d/2}$  induces a surjection in $\G$-equivariant rational cohomology for $\tau={\tau_{max}}$.  Then the same is true for all $\tau\in (d/2,d)$.  Moreover,  $\dim H^p_\G(X_{\tau-d/2}, X'_{\tau-d/2})$ is independent of $\tau$ for all $p$.
\end{lemma}

\begin{proof}
The sets $X'_{\tau-d/2}$, $X_{\tau-d/2}$ remain unchanged for $\tau$  in a  connected component of $(d/2,d)\setminus C_d$, 
where $C_d$ is given in \eqref{eqn:cd}.
Fix $\tau_c\in C_{d}$, $2\tau_c-d/2=k\in \ZBbb$, and let 
$\tau_l<\tau_c<\tau_r$ be in components $(d/2,d)\setminus C_{d}$ containing $\tau_c$ in their closures. 
Let $\delta^{l,r}=2\tau_c-d/2-\tau_{l,r}$.  Note that $\delta^{l,r}\in \Delta^-_{\tau_{l,r},d}$, $\delta^l>\tau_l-d/2$, and $\delta^r<\tau_r-d/2$. 
 Also,  we claim
\begin{equation} \label{eqn:claim}
X_{\tau_r-d/2}=X_{\tau_l-d/2}\cup \B_{\delta^l}^{\tau_l}\ ,\
X'_{\tau_r-d/2}=X'_{\tau_l-d/2}\cup \B_{\delta^l}^{\tau_l}
\end{equation}
To see this, we refer to Figure 1 and the discussion preceding it.  Under the map $\Delta_{\tau_l,d}\to \Delta_{\tau_r,d}$, $\delta^l\mapsto \delta^r$ and
$\tau_l-d/2\mapsto \tau_r-d/2$.  
 The claim then follows if we show that $\delta^r$ is the predecessor of $\tau_r-d/2$ in $\Delta_{\tau_r,d}$, and
$\delta^l$ is the successor of $\tau_l-d/2$ in $\Delta_{\tau_l,d}$ (see Figure 1).
So suppose $\delta\in \Delta_{\tau_r,d}$, $\delta<\tau_r-d/2$.  By Remark \ref{rem:gap}, we may assume $\delta\in \Delta^-_{\tau_r,d}$. Write
$\delta+\tau_r=\ell\in \ZBbb$.  Then $\ell\leq 2\tau_r-d/2$, which implies $\ell\leq k$, and $\delta\leq \delta_r$.  The reasoning is similar for
$\delta_l$.

Now since the result holds by assumption for $\tau_{max}$, we may assume by induction that the result holds for $\tau\geq \tau_r$.
Then we have 
\begin{equation}
\xymatrix{
0 \ar[r] & H_{\G}^p( X_{\tau_r-d/2},  X'_{\tau_r-d/2}) \ar[d]^{f} \ar[r] &
H_{\G}^p( X_{\tau_r-d/2}) \ar[r] \ar[d]^{g} & H_{\G}^p( X'_{\tau_r-d/2}) \ar[d]^{h} \ar[r] & 0\\
\cdots \ar[r] & H_{\G}^p(X_{\tau_l-d/2},  X'_{\tau_l-d/2})  \ar[r] &
 H_{\G}^p(X_{\tau_l-d/2})\ar[r]  & H_\G^p(X'_{\tau_l-d/2}) \ar[r] & \cdots
}
\end{equation}
By \eqref{eqn:claim} and the proof of Lemma \ref{lem:delta}, $h$ is surjective.  Hence, the lower long exact sequence must split.
Moreover, $g$ is surjective as well, and $\ker g=\ker h$.  As a consequence,  $f$ must be an isomorphism.
  The result now follows by induction.
\end{proof}

\begin{proof}[Proof of Theorems \ref{thm:perfect} and \ref{thm:kirwan}]
We proceed   by induction as follows.
  First, if $d>4g-4$, then by Lemma \ref{lem:claim}, 
the hypothesis of Lemma \ref{lem:induction} is satisfied.  
It then follows from Lemma \ref{lem:anytau} that the 
inclusion $X'_{\tau-d/2}\hookrightarrow  X_{\tau-d/2}$  induces a 
surjection in $\G$-equivariant rational cohomology for any $\tau$.   In 
particular, this is true for the value $\tilde\tau_{max}$ corresponding 
to degree $d-2$.  Hence, the inductive hypothesis holds, and 
the result is proven for all $d$.
Kirwan surjectivity follows immediately.
\end{proof}

\begin{proof}[Proof of Theorem \ref{thm:macdonald}]
This follows from Kirwan surjectivity, but more generally we prove this on each stratum.
Clearly it suffices to prove the result for $k=1$.
Since the gauge groups for $E$ and $\widetilde E$ are canonically isomorphic, 
it suffices by induction to show 
that if the result holds for the inclusion $X_\delta\hookrightarrow 
\widetilde X_\delta$, then it also holds for $X_{\delta_1}\hookrightarrow 
\widetilde X_{\delta_1}$, where $\delta_1$ is the predecessor of 
$\delta$ in $\Delta_{\tau,d}$.  By Theorem \ref{thm:perfect}, the 
diagram \eqref{eqn:embedding} holds for \emph{all} $\delta$.  
It follows that if $g$ is surjective, then so is $h$.   This completes the proof.
\end{proof}

\section{Cohomology of moduli spaces}

\subsection{Equivariant cohomology of $\tau$-semistable pairs}

The purpose of this section is to complete the calculation of 
the $\mathcal{G}$-equivariant Poincar\'e polynomial of $\mathcal{B}^\tau_{ss}$.
First we consider the case where $\tau$ is generic. Choose an integer $N$,
$d/2 < N \leq d$, and let
$ \tau \in  \left(\max\{d/2, N-1\}, N \right)$.
Then the different allowable values of $\delta$ for each type of stratum and the cohomology  are as follows (see Proposition \ref{prop:critical_set_cohomology}).

\begin{itemize}

\item[($\Ia$)] 
 There is one stratum 
$\Ia^{d/2}$ corresponding to $\mathcal{A}_{ss}$ (indexed by $j = d/2$), 
and by Lemma \ref{lem:anytau} the contribution $\Ia^{d/2}(t)$ to 
the Poincar\'e polynomial is independent of $\tau$.  For $d>4g-4$ it follows from
\eqref{eqn:minimal_ym} that
\begin{equation}\label{eqn:iad}
\Ia^{d/2}(t) = \frac{t^{2d+4-4g}}{(1-t^2)} P_t^{\overline\G}(\A_{ss}) 
\end{equation}
where $\overline \G$ is defined in \cite[p.\ 577]{AtiyahBott83}.
  We compute $\Ia^{d/2}(t)$ in general in Lemma \ref{lem:Ia} below.
 The remaining strata are indexed by integers $j = j(\delta) = \mu_+$ such that $d/2 < j \leq N-1$ and $\delta = j - d + \tau$.
The contribution to the $\mathcal{G}$-equivariant Poincar\'e polynomial  is 
\begin{align}
\begin{split}
\Ia^j(t) & = \frac{t^{2(2 j(\delta) - d + g - 1)}}{(1-t^2)^2} P_t \left( J_{j(\delta)}(M) \times J_{d-j(\delta)}  (M) \right) - \frac{t^{2j(\delta)}}{1-t^2} P_t \left( S^{j(\delta)} M \times J_{d-j(\delta)}(M) \right) \\
& \quad \quad \quad \quad -
 \frac{t^{2(2j(\delta)-d+g-1)}}{1-t^2} P_t \left( S^{d-j(\delta)} M \times J_{j(\delta)}(M) \right) 
\end{split}  \label{eqn:Ia-unstable}
\end{align}
\item[($\Ib$)] There are an infinite number of strata indexed by
 integers $j = j(\delta) = \mu_+$ such that $N \leq j$ and $\delta = j-d+\tau$.
The contribution  is 
\begin{align}
\begin{split}
\Ib^j(t) &= \frac{t^{2(2j(\delta) - d+g-1)}}{(1-t^2)^2} P_t \left( J_{j(\delta)}(M) \times J_{d-j(\delta)}(M) \right)\\
 &\qquad\qquad - \frac{t^{2(2j(\delta) - d+g-1)}}{1-t^2} P_t \left( S^{d-j(\delta)} M \times J_{j(\delta)} (M) \right) 
\end{split} \label{eqn:Ib}
\end{align}
\item[($\IIplus$)]  These strata are indexed by integers $j = j(\delta) = \deg \Phi = \deg L_1$ 
such that $d-N +1 \leq j \leq N-1$, and $\delta = j - d + \tau$. The contribution  is 
\begin{equation}\label{eqn:IIplus}
\IIplus_j(t) = \frac{t^{2 j(\delta)}}{1-t^2} P_t \left( S^{j(\delta)} M \times J_{d-j(\delta)} (M) \right) 
\end{equation}
\item[($\IIminus$)] These strata are indexed by integers $j=d-j(\delta)$ 
such that $0 \leq j \leq d-N$, where $\delta = j(\delta)-\tau = d-j-\tau$, and 
the contribution 
is
\begin{equation}\label{eqn:IIminus}
\IIminus_j(t) = \frac{t^{2(2j(\delta)-d+g-1)}}{1-t^2} P_t \left( S^{d-j(\delta)} M \times J_{j(\delta)} (M) \right) 
\end{equation}
\end{itemize}

Then we have
\begin{theorem} \label{thm:poincare-poly-generic}
For  $\tau \in  \left(\max\{d/2, N-1\}, N \right)$, 
$$P_t(\M_{\tau,d})=
P_t^\mathcal{G} (\mathcal{B}^\tau_{ss}) 
= P_t(B \mathcal{G}) - \Ia^{d/2}(t) - \sum_{j = \lfloor d/2 + 1 \rfloor}^{N-1} \Ia^j (t) - \sum_{j=N}^\infty \Ib^j(t) \\
 - \sum_{j=0}^{d-N} \IIminus_j(t) - \sum_{j = d-N+1}^{N-1} \IIplus_j(t)
 $$
\end{theorem}

\begin{proof}
By Theorem \ref{thm:perfect} we have
$$
P_t^\mathcal{G} (\mathcal{B}^\tau_{ss}) 
= P_t(B \mathcal{G}) - 
\sum_{\delta\in \Delta_{\tau,d}\setminus\{0\}} P_t^\G(X_\delta, X_{\delta_1})
$$
If $\delta\not\in \Delta_{\tau,d}^+\cap (\tau-d/2, \tau]$, then by the Morse-Bott lemma \eqref{eqn:bott1} and \eqref{eqn:thom_eta},
$$
P_t^\G(X_\delta, X_{\delta_1})
=\frac{t^{2\sigma(\delta)}}{1-t^2}P^\G_t(\eta_\delta)
$$
where $\sigma(\delta)$ is given in Corollary \ref{cor:smooth}.  If $\delta\in \Delta_{\tau,d}^+\cap
 (\tau-d/2, \tau]$ then by Section \ref{sec:large_degree},
$$
P_t^\G(X_\delta, X_{\delta_1})
= P_t^\G(X_\delta, X_{\delta}'')- P_t^\G(X_{\delta_1}, X_\delta'')
$$
The first term on the right hand side is given by \eqref{eqn:ab}.  For the second term, we have
$$
H^\ast_\G(X_{\delta_1}, X_\delta'')=\begin{cases}
H^\ast_\G(\nu_{I,\delta}', \nu_{I,\delta}'') & \delta\in \Delta^+_{\tau,d}\cap [2\tau-d,\tau] \\
H^\ast_\G(\omega_{\delta}, \nu_{I,\delta}'') & \delta\in \Delta^+_{\tau,d}\cap (\tau-d/2,2\tau-d) 
\end{cases}
$$
and the latter cohomology groups have been computed in \eqref{eqn:nuprime1} and 
\eqref{eqn:omegadoubleprime}.
This completes the computation.
\end{proof}

When the parameter $\tau$ is non-generic (i.e. $\tau = N$ for some integer $N \in \left[ d/2, d \right]$) then 
the same analysis as above applies, however now there are split solutions to the vortex equations. These correspond 
to one of the critical sets of type ${\bf II}$, where $E = L_1 \oplus L_2$ with $\phi \in H^0(L_1) \setminus \{ 0 \}$, and $\deg 
L_2 = \tau$. Therefore, the only difference the generic and non-generic case is that we do not count any contribution from 
the critical set of type $\IIminus$ with $j = d-N$. Therefore the Poincar\'e polynomial is
\begin{theorem}\label{thm:poincare-poly-non-generic}
For $\tau=N$,
$$P_t^\mathcal{G} (\mathcal{B}^N_{ss}) = P_t(B \mathcal{G}) - \Ia^{d/2}(t) - \sum_{j = \lfloor d/2 + 1 \rfloor}^{N-1} \Ia^j (t) 
- \sum_{j=N}^\infty \Ib^j(t) \\
 - \sum_{j=0}^{d-N-1} \IIminus_j(t) - \sum_{j = d-N+1}^{N-1} \IIplus_j(t) 
 $$
\end{theorem}

Finally, using Theorem \ref{thm:poincare-poly-generic},  we can give
 a computation of the remaining term which is as yet undetermined in low degree.
\begin{lemma} \label{lem:Ia}
For all $d\geq 2$, 
\begin{align*}
\Ia^{d/2}(t)&=
\frac{1}{1-t^2} P_t^{\overline\G}(\A_{ss})-\sum_{j=0}^{\lfloor d/2\rfloor}
\frac{t^{2j}-t^{2(d+g-1-2j)}}{1-t^2}P_t(S^jM\times J_{d-j}(M)) \\
&\qquad-
\begin{cases}0&\text{ if $d$ odd} \\ 
\displaystyle \frac{t^{2g-2}}{(1-t^2)}P_t(S^{d/2}M\times J_{d/2}(M)) &\text{ if $d$ even}
\end{cases}
\end{align*}
\end{lemma}

\begin{remark}
It can be verified directly  that for $d>4g-4$, the expression above 
agrees with \eqref{eqn:iad}.
 See the argument of Zagier in  \cite[pp.\ 336-7]{Thaddeus94}.
\end{remark}

\begin{proof}[Proof of Lemma \ref{lem:Ia}]
Take the special case $N=d$. 
  Then $\M_{\tau,d}$ is a projective bundle over $J_d(M)$, and so 
$$
P_t(\M_{\tau,d})=\frac{ 1-t^{2(d+g-1)}}{1-t^2}P_t(J_{d}(M))
$$
On the other hand, from Theorem 4.1 we have
$$
P_t(\M_{\tau,d})=P_t(B\G)-\Ia^{d/2}(t)-\sum_{j=\lfloor d/2+1\rfloor}^{d-1} \Ia^j(t)-
\sum_{j=d}^\infty \Ib^j(t) -\IIminus_0(t)-\sum_{j=1}^{d-1} \IIplus_j(t)
$$
Now notice that the term $\IIminus_0(t)$ is cancelled by the second term in $\Ib^d$.  We combine
the remaining terms in the sum of $\Ib^j$ with the sum of $\Ia^j$. 
We have 
\begin{align*}
P_t(\M_{\tau,d})&=P_t(B\G)-\Ia^{d/2}(t)-\sum_{j=\lfloor d/2+1\rfloor}^\infty
 \frac{t^{2(2j-d+g-1)}}{(1-t^2)^2}P_t(J_j(M)\times J_{d-j}(M)) \\
&\qquad +\sum_{j=\lfloor d/2+1\rfloor}^{d-1} 
\frac{t^{2j}}{(1-t^2)}P_t(S^jM\times J_{d-j}(M))\\
&\qquad +\sum_{j=\lfloor d/2+1\rfloor}^{d-1}\frac{t^{2(2j-d+g-1)}}{(1-t^2)}
P_t(S^{d-j}M\times J_{j}(M)) 
 \\
&\qquad  
-\sum_{j=1}^{d-1} \frac{t^{2j}}{(1-t^2)}P_t(S^jM\times J_{d-j}(M)) \\
&= P_t(B\G)-\Ia^{d/2}(t)-\sum_{j=\lfloor d/2+1\rfloor}^\infty
 \frac{t^{2(2j-d+g-1)}}{(1-t^2)^2}P_t(J_j(M)\times J_{d-j}(M)) \\
&\qquad +\sum_{j=\lfloor d/2+1\rfloor}^{d-1} 
\frac{t^{2(2j-d+g-1)}}{(1-t^2)} P_t(S^{d-j}M\times J_{j}(M)) \\
&\qquad-\sum_{j=1}^{\lfloor d/2\rfloor}\frac{t^{2j}}{(1-t^2)}P_t(S^jM\times J_{d-j}(M))
\end{align*}
Now make the substitution $j\mapsto d-j$ in the second to the last sum, using
$$
d-\lfloor d/2+1\rfloor=\begin{cases} d/2-1=\lfloor d/2\rfloor-1 &\text{ if $d$ even} \\
d/2-1/2=\lfloor d/2\rfloor &\text{ if $d$ odd}
\end{cases}
$$
The result now follows from this,  \cite[Thm.\ 7.14]{AtiyahBott83}, and the fact that 
$
P_t(\M_{\tau,d})
$
is equal to  the $j=0$ term in the sum.
\end{proof}

\subsection{Comparison with the results of Thaddeus}

In \cite{Thaddeus94}, Thaddeus computed the Poincar\'e polynomial of the moduli space using  different methods 
to those of this paper. The idea is to show that when the parameter $\tau$ passes a critical value, then the moduli space $\M_{\tau,d}$ undergoes a 
birational transformation consisting of a blow-down along a submanifold and a blow-up along a different submanifold (these transformations are 
known as ``flips''). By computing the change in Poincar\'e polynomial caused by the flips as the parameter crosses the critical values, 
and also observing that the moduli space is a projective space for one extreme value of $\tau$, Thaddeus computed the Poincar\'e polynomial 
of the moduli space for any value of the parameter.  In this section we recover this result from Theorem \ref{thm:poincare-poly-generic}.
In the Morse theory picture we see that the critical point structure 
changes: As $\tau$ increases past a critical value then a new critical set  appears, and the index may change at 
existing critical points.

\begin{theorem}\label{thm:flip-change}
Let $N\in \ZBbb$, $d/2<N\leq d-1$.
Then for 
$\tau \in \left( \max(d/2, N-1), N \right)$,
\begin{equation}\label{eqn:flip-change}
P_t(\M_{\tau+1, d}) - P_t(\M_{\tau, d}) = \frac{t^{4N-2d+2g-2} - t^{2d-2N}}{1-t^2} P_t \left( S^{d-N} M \times J_N(M) \right)
\end{equation}
As a consequence, the Poincar\'e polynomial of the moduli space has the form
\begin{equation}\label{eqn:overall-polynomial}
P_t(\M_{\tau,d})=
\frac{(1+t)^{2g}}{1-t^2}\text{\rm Coeff}_{x^N}\left(\frac{t^{2d+2g-2-4N}}{xt^4-1}-
\frac{t^{2N+2}}{x-t^2}\right)\left(\frac{(1+xt)^{2g}}{(1-x)(1-xt^2)}\right) 
\end{equation}
\end{theorem}

\begin{remark}
Let $\M_{\tau,d}^0$ denote the moduli space where the bundle has fixed determinant (see \cite{Thaddeus94}). 
The analysis in this paper applies in this case as well.  In particular, one obtains
\begin{equation}
P_t(\M_{\tau+1, d}^0) - P_t(\M_{\tau, d}^0) = \frac{t^{4N-2d+2g-2} - t^{2d-2N}}{1-t^2} P_t \left( S^{d-N} M \right) 
\end{equation}
This exactly corresponds to Thaddeus' results for $P_t(\mathbb{P} W_j^+) - P_t(\mathbb{P} W_j^-)$  \cite[p.\ 21]{Thaddeus94}, where $j = d-N$.
\end{remark}

\begin{proof}[Proof of Theorem \ref{thm:flip-change}]
By Theorem \ref{thm:poincare-poly-generic},
\begin{equation}\label{eqn:polynomial-difference}
P_t(\M_{\tau+1, d}) - P_t(\M_{\tau, d}) = - \Ia^N(t) + \Ib^N(t) + \IIminus_{d-N}(t) - \IIplus_{d-N}(t) - \IIplus_N(t) 
\end{equation}
Substituting in the results of \eqref{eqn:Ia-unstable}, \eqref{eqn:Ib}, \eqref{eqn:IIminus}, and \eqref{eqn:IIplus} gives
\begin{align*}
P_t(\M_{\tau+1, d}) - P_t(\M_{\tau, d})
 & = -\frac{t^{2(2 N - d + g - 1)}}{(1-t^2)^2}
 P_t \left( J_{N}(M) \times J_{d-N}  (M) \right) + \frac{t^{2N}}{1-t^2} P_t \left( S^{N} M \times J_{d-N}(M) \right) \\
& \quad \quad \quad \quad + \frac{t^{2(2N-d+g-1)}}{1-t^2} P_t \left( S^{d-N} M \times J_{N}(M) \right) \\
&  \quad \quad + \frac{t^{2(2N - d+g-1)}}{(1-t^2)^2} P_t \left( J_{N}(M) \times J_{d-N}(M) \right) \\
& \quad \quad \quad \quad - \frac{t^{2(2N - d+g-1)}}{1-t^2} P_t \left( S^{d-N} M \times J_{N} (M) \right) \\
& \quad \quad + \frac{t^{2(d-2(d-N)+g-1)}}{1-t^2} P_t \left( S^{d-N} M \times J_{N} (M) \right) \\
& \quad \quad - \frac{t^{2 (d-N)}}{1-t^2} P_t \left( S^{d-N} M \times J_{N} (M) \right) - \frac{t^{2 N}}{1-t^2} P_t \left( S^{N} M \times J_{d-N} (M) \right) \\
& = \frac{1}{1-t^2} P_t \left( S^{d-N} M \times J_N(M) \right) \left( t^{4N-2d+2g-2} - t^{2d-2N} \right)
\end{align*}
as required.
Using the results of \cite{MacDonald62} on the cohomology of the symmetric product, and the fact that $P_t(J_N(M)) = (1+t)^{2g}$, we see that the same method as for the proof of \cite[(4.1)]{Thaddeus94} gives equation \eqref{eqn:overall-polynomial}.
\end{proof}

\begin{remark}
For $\tau$ as above,
Theorem \ref{thm:poincare-poly-non-generic} shows that the difference
\begin{equation*}
P_t^\mathcal{G} (\mathcal{B}^{N}_{ss}) - P_t(\M_{\tau,d}) = \IIminus_{d-N}(t) = \frac{t^{4N-2d+2g-2}}{1-t^2} P_t \left( S^{d-N} M \times J_N(M) \right)
\end{equation*}
comes from only one critical set; the type ${\bf II}$ critical 
set corresponding to a solution of the vortex equations when $\tau = N$. 
 The rest of the terms in \eqref{eqn:polynomial-difference}, corresponding to the difference
$$
P_t( \M_{\tau+1,d}) - P_t^\mathcal{G}(\mathcal{B}^{N}_{ss})  
= - \Ia^N(t) + \Ib^N(t) - \IIplus_{d-N}(t) - \IIplus_N(t) 
  = -\frac{t^{2N}}{1-t^2} P_t \left( S^N M \times J_{d-N}(M) \right)
$$
come from a number of changes that occur in the structure of 
the critical sets as $\tau$ increases past $N$: 
the term $-\IIplus_{d-N}(t)$ corresponds to the type ${\bf II}$ critical point that no 
longer is a solution to the vortex equations, the term $-\IIplus_N(t)$ corresponds to the new critical 
point of type $\IIplus$ that appears, and the term $-\Ia^N(t)+\Ib^N(t)$ corresponds to the critical point that changes 
type from $\Ib$ to $\Ia$.

Therefore we see that the changes in the critical set structure as $\tau$ crosses the 
critical value $N$ are localized to two regions of $\mathcal{B}$. The first corresponds to interchanging critical sets
of type $\IIminus$ and type $\IIplus$.  This is the phenomenon illustrated in Figure 1.
The second corresponds to critical sets  
of  type $\Ia$ and $\IIplus$
that merge to form a single component of type $\Ib$.
 The terms 
from the first change exactly correspond to those in \eqref{eqn:flip-change}, i.e.
\begin{align*}
\IIminus_{d-N}(t) - \IIplus_{d-N}(t) & = \frac{t^{4N-2d+2g-2} - t^{2d-2N}}{1-t^2} P_t \left( S^{d-N} M \times J_N(M) \right) \\
 & = P_t(\M_{\tau+1, d}) - P_t(\M_{\tau, d}) 
\end{align*}
and the terms from the second change cancel each other, i.e. 
$ \Ib^N(t) -\Ia^N(t) - \IIplus_N(t) = 0
$.
\end{remark}

$\hbox{}$
\bibliographystyle{plain}
\bibliography{ref}

\end{document}